\documentclass[11pt]{amsart}

\usepackage{enumerate}
\usepackage{amsmath,amsthm,verbatim,amssymb,amsfonts,amscd, graphicx}
\usepackage{tikz-cd} 
\usepackage{graphics}
\usepackage{hyperref}
\usepackage{mathrsfs}
\usepackage{relsize}
\usepackage{enumitem}
\usepackage[utf8]{inputenc}
\usepackage{csquotes}
\usepackage{amsthm}
\usepackage{thmtools}
\usepackage[sort cites=true, backend=biber]{biblatex}
\usepackage{biblatex}

\tikzset{
	symbol/.style={
		draw=none,
		every to/.append style={
			edge node={node [sloped, allow upside down, auto=false]{$#1$}}}
	}
}

\begin{document}

	\title{Minimal pairs, minimal fields and implicit constant fields}
	\author{Arpan Dutta}
	\def\NZQ{\mathbb}               
	\def\NN{{\NZQ N}}
	\def\QQ{{\NZQ Q}}
	\def\ZZ{{\NZQ Z}}
	\def\RR{{\NZQ R}}
	\def\CC{{\NZQ C}}
	\def\AA{{\NZQ A}}
	\def\BB{{\NZQ B}}
	\def\PP{{\NZQ P}}
	\def\FF{{\NZQ F}}
	\def\GG{{\NZQ G}}
	\def\HH{{\NZQ H}}
	\def\UU{{\NZQ U}}
	\def\P{\mathcal P}
	
	%
	%
	\let\union=\cup
	\let\sect=\cap
	\let\dirsum=\oplus
	\let\tensor=\otimes
	\let\iso=\cong
	\let\Union=\bigcup
	\let\Sect=\bigcap
	\let\Dirsum=\bigoplus
	\let\Tensor=\bigotimes
	
	\theoremstyle{plain}
	\newtheorem{Theorem}{Theorem}[section]
	\newtheorem{Lemma}[Theorem]{Lemma}
	\newtheorem{Corollary}[Theorem]{Corollary}
	\newtheorem{Proposition}[Theorem]{Proposition}
	\newtheorem{Problem}[Theorem]{}
	\newtheorem{Conjecture}[Theorem]{Conjecture}
	\newtheorem{Question}[Theorem]{Question}
	
	\theoremstyle{definition}
	\newtheorem{Example}[Theorem]{Example}
	\newtheorem{Examples}[Theorem]{Examples}
	\newtheorem{Definition}[Theorem]{Definition}
	
	\theoremstyle{remark}
	\newtheorem{Remark}[Theorem]{Remark}
	\newtheorem{Remarks}[Theorem]{Remarks}
	
	\newcommand{\trdeg}{\mbox{\rm trdeg}\,}
	\newcommand{\rr}{\mbox{\rm rat rk}\,}
	\newcommand{\sep}{\mathrm{sep}}
	\newcommand{\ac}{\mathrm{ac}}
	\newcommand{\ins}{\mathrm{ins}}
	\newcommand{\res}{\mathrm{res}}
	\newcommand{\Gal}{\mathrm{Gal}\,}
	\newcommand{\ch}{\mathrm{char}\,}
	\newcommand{\Aut}{\mathrm{Aut}\,}
	\newcommand{\kras}{\mathrm{kras}\,}
	\newcommand{\dist}{\mathrm{dist}\,}
	\newcommand{\ord}{\mathrm{ord}\,}
	
	\newcommand{\n}{\par\noindent}
	\newcommand{\nn}{\par\vskip2pt\noindent}
	\newcommand{\sn}{\par\smallskip\noindent}
	\newcommand{\mn}{\par\medskip\noindent}
	\newcommand{\bn}{\par\bigskip\noindent}
	\newcommand{\pars}{\par\smallskip}
	\newcommand{\parm}{\par\medskip}
	\newcommand{\parb}{\par\bigskip}
	\let\epsilon\varepsilon
	\let\phi=\varphi
	\let\kappa=\varkappa
	
	\def \a {\alpha}
	\def \b {\beta}
	\def \s {\sigma}
	\def \d {\delta}
	\def \g {\gamma}
	\def \o {\omega}
	\def \l {\lambda}
	\def \th {\theta}
	\def \D {\Delta}
	\def \G {\Gamma}
	\def \O {\Omega}
	\def \L {\Lambda}
	%
	%
	\textwidth=15cm \textheight=22cm \topmargin=0.5cm
	\oddsidemargin=0.5cm \evensidemargin=0.5cm \pagestyle{plain}

\address{Department of Mathematics, IISER Mohali,
		Knowledge City, Sector 81, Manauli PO,
		SAS Nagar, Punjab, India, 140306.}
\email{arpan.cmi@gmail.com}
	
\date{\today}

\thanks{}


	
\maketitle
	
	\begin{abstract}
		Minimal pairs of definition were introduced by Alexandru, Popescu and Zaharescu [\ref{AP sur une classe}, \ref{APZ characterization of residual trans extns}, \ref{APZ2 minimal pairs}] to study residue transcendental extensions. In this paper we obtain analogous results in the value transcendental case. We introduce the notion of minimal fields of definition for valuation transcendental extensions and show that they share some common ramification theoretic properties. The connection between minimal fields of definition and implicit constant fields is also investigated. Further, we explore the relationship between valuation transcendental extensions and pseudo-Cauchy sequences. 
	\end{abstract}
	
	\section{Introduction}
	
	Given a valued field $(K,\nu)$, it is of interest to understand the set of all possible extensions of $\nu$ to a rational function field $K(x)$. On the other hand, given an extension $\o$ of $\nu$ to $K(x)$, it is important to give a complete description of the valuation $\o$. These two problems are tightly connected and several objects have been introduced to tackle such problems. The notion of minimal pairs of definition was introduced by Alexandru, Popescu and Zaharescu to study residue transcendental extensions [\ref{AP sur une classe}, \ref{APZ characterization of residual trans extns}, \ref{APZ2 minimal pairs}]. MacLane [\ref{MacLane key pols}] introduced the notion of key polynomials, which was then generalized by Vaqui\'{e} in [\ref{Vaquie key pols}]. An alternative form of key polynomials was given by Spivakovsky and Novacoski in [\ref{Nova Spiva key pol pseudo convergent}]. The notion of pseudo-Cauchy sequences was introduced by Kaplansky in his celebrated paper [\ref{Kaplansky}]. All the relevant definitions will be provided in Section \ref{section preliminiaries}.
	
	\pars An extension $\o$ of a valuation $\nu$ on a field $K$ to a rational function field $K(x)$ satisfies the famous Abhyankar inequality:
	\[ \rr \o K(x)/\nu K  + \trdeg [K(x)\o:K\nu] \leq 1, \]
	where $\o K(x)$ and $\nu K$ denote the correpsonding value groups, $K(x)\o$ and $K\nu$ denote the corresponding residue fields and $\rr \o K(x)/\nu K $ is the $\QQ$-dimension of the divisible hull $\frac{\o K(x)}{\nu K}\tensor_{\ZZ}\QQ$. This inequality is a consequence of Theorem 1 of [\S 10.3, Chapter VI, \ref{Bourbaki}]. The extension $\o$ is said to be \textbf{value transcendental} if $\rr \o K(x)/\nu K = 1$ and \textbf{residue transcendental} if $\trdeg[K(x)\o:K\nu]=1$. The extension $\o$ is said to be \textbf{valuation transcendental} if it is either value or residue transcendental. In Section \ref{section min pair}, we extend the results obtained in [\ref{AP sur une classe}, \ref{APZ characterization of residual trans extns}, \ref{APZ2 minimal pairs}] to the case of value transcendental extensions. In particular, we show that a value transcendental extension is completely described by a minimal pair of definition [cf. Section \ref{section preliminiaries}]. We also consider the question: given a pair of definition $(a,\g)$ for a valuation transcendental extension $\o$, can we find a minimal pair of definition for it? Using the notion of complete distinguished chains, introduced by Popescu and Zaharescu in [\ref{PZ structure of irr poly over local fields}], we provide an answer to the above question in Theorem \ref{Thm min pair comp dist chain}, under the assumption that $K(a)|K$ is a defectless extension of henselian valued fields [cf. Section \ref{section preliminiaries}]. Further, in Theorem \ref{thm extensions coincide on K(x)} we show that for a valuation transcendental extension $\o$ of $\nu$ to $K(x)$ and an extension $\overline{\nu}$ of $\nu$ to $\overline{K}$, the simultaneous extensions of $\o$ and $\overline{\nu}$ to $\overline{K}(x)$ are tightly connected, where $\overline{K}$ is a fixed algebraic closure of $K$. As a consequence we obtain that there are a finite number of such common extensions. Although some of the results obtained in this section are already known, we recreate them here for the sake of completeness.

	\pars In Section \ref{section min fields} we introduce the notion of minimal fields of definition. A minimal field of definition for a valuation transcendental extension $\o$ is defined to be of the form $K(a)$, where $(a,\g)$ is a minimal pair of definition for $\o$. We show that the minimal fields of definition for $\o$ share some common ramification theoretic properties. The relationship between implicit constant fields, introduced by Kuhlmann in [\ref{Kuh value groups residue fields rational fn fields}], and the minimal fields of definition is also explored. For a valued field extension $(K(x)|K,\o)$ and an extension of $\o$ to $\overline{K(x)}$, the \textbf{implicit constant field} is defined to be the relative algebraic closure of $K$ in the henselization of $K(x)$ and is denoted by $IC(K(x)|K,\o)$. We obtain the following result in this paper:
	
	\begin{Theorem}\label{Thm IC and min fields}
		Take a valued field $(K,\nu)$, a valuation transcendental extension $\o$ of $\nu$ to $K(x)$, and a minimal field of definition $K(a)$ for $\o$. Take an extension $\overline{\o}$ of $\o$ to $\overline{K}(x)$ which has $(a,\g)$ as a minimal pair of definition for some $\g\in\overline{\o}\overline{K}(x)$, and fix an extension of $\overline{\o}$ to $\overline{K(x)}$. Take $b_1, b_2 \in\overline{K}$ such that $K(b_1) = K(a)\sect K^i$ and $K(b_2) = K(a)\sect K^r$, where $K^i$ and $K^r$ denote the absolute inertia field and the absolute ramification field of $(K,\nu)$. Write the henselizations of $K(a)$ and $K(b_i)$ as $K(a)^h$ and $K(b_i)^h$.
		
		If $\o$ is residue transcendental, then
		\begin{equation}\label{eqn r.t. IC and min fields}
			K(b_1)^h \subseteq IC(K(x)|K,\o)\subseteq K(a)^h.
		\end{equation}
		 
		If $\o$ is value transcendental, then
		\begin{equation}\label{eqn IC and min fields}
			K(b_2)^h \subseteq IC(K(x)|K,\o)\subseteq K(a)^h.
		\end{equation}
	\end{Theorem}
	An important relevant problem is the following:
	\begin{Question}\label{question}
	Take notations and assumptions as in Theorem \ref{Thm IC and min fields}. Give an explicit computation for $IC(K(x)|K,\o)$. 
	\end{Question}
	The absolute inertia field and the absolute ramification field are separable-algebraic extensions of $K$. The problem of explicit computation of the implicit constant field when $K$ admits an arbitrary separable-algebraic minimal field of definition for $\o$ remains open. A partial solution to that problem is provided in the following theorem:
	
	\begin{Theorem}\label{Thm gamma and kras}
		Let notations and asumptions be as in Theorem \ref{Thm IC and min fields}. Assume that $a$ is separable over $K$. Then the following statements hold true:
		\sn (i) Assume that $\g>\kras(a,K)$. Then $IC(K(x)|K,\o) = K(a)^h$.
		\n (ii) Assume that $\o$ is value transcendental, $\nu$ admits a unique extension from $K$ to $K(a)$ and $\g<\kras(a,K)$. Then $(\overline{\o}K(a,x):\o K(x)) =j$ where $\overline{\o}(a-a_i) > \g$ for exactly $j$ many conjugates $a_i$ of $a$, including $a$ itself. Consequently, $IC(K(x)|K,\o) \subsetneq K(a)^h$.
		\n (iii) Assume that $\o$ is residue transcendental, $\nu$ admits a unique extension from $K$ to $K(a)$ and $\g\leq\kras(a,K)$. Then either $\overline{\o}K(a,x)\neq \o K(x)$, or, $[K(a,x)\overline{\o}:K(x)\o] = j$ where $\overline{\o}(a-a_i) \geq \g$ for exactly $j$ many conjugates $a_i$ of $a$, including $a$ itself. Consequently, $IC(K(x)|K,\o) \subsetneq K(a)^h$.
	\end{Theorem}
	In particular, Theorem \ref{Thm gamma and kras} provides a partial solution to Question \ref{question} when $(K,\nu)$ is either henselian or has rank one (that is, $\nu K$ is an ordered subgroup of $\RR$), as observed in Corollary \ref{Corollary IC when K henslian or rank one}.
	
	 On the other hand, if there is a minimal field of definition for $\o$ that is a purely inseparable extension of $K$, then we obtain that $IC(K(x)|K,\o)$ equals the henselization of $K$. The problem remains open when a minimal field of definition for $\o$ is an arbitrary inseparable extension of $K$.
	\newline Finally, in Theorem \ref{Thm gamma > vK} we provide a complete solution to Question \ref{question} when $\o$ admits a unique pair of definition. We show that in that situation the implicit constant field is the separable closure of $K$ in the henselization of the minimal field of definition for $\o$.
	
	\pars In Section \ref{section examples} we provide several examples illustrating that the ramification or field theoretic structure of the minimal field of definition may not necessarily be recovered from the information of the implicit constant field. Examples showing that the implicit constant field may be a proper non-trivial subextension of the henselization of a minimal field of definition for $\o$ are also constructed.
	
	\pars Section \ref{section pseudo cauchy} deals with the connection between valuation transcendental extensions and pseudo-Cauchy sequences. Specifically, we consider the case when $x$ is a limit of a pseudo-Cauchy sequence $(z_\mu)_{\mu<\l}$ in $(K,\nu)$, and study its connection with the pairs of definition and key polynomials for $\o$. 
	
	\pars For basic information on valuation theory, we refer the reader to [\ref{Endler book}, \ref{Engler book}, \ref{ZS2}].
	
	\section*{Acknowledgements}
	 This work was supported by the Post-Doctoral Fellowship of the National Board of Higher Mathematics, India. The author would like to thank Franz-Viktor Kuhlmann for patiently reading multiple versions of the paper, and for his many insightful suggestions.
	
	
	\section{Preliminaries}\label{section preliminiaries}
	
		We adopt the following convention in the present paper: $(K,\nu)$ denotes a field $K$ equipped with a valuation $\nu$, $\o$ denotes some valuation transcendental extension of $\nu$ to $K(x)$, $\overline{\nu}$ denotes some extension of $\nu$ to a fixed algebraic closure $\overline{K}$ of $K$, and $\overline{\o}$ denotes some common extension of $\o$ and $\overline{\nu}$ to $\overline{K}(x)$. The existence of such a common extension is guaranteed by Lemma 3.2 of [\ref{Kuh value groups residue fields rational fn fields}]. However, $\overline{\o}$ is not necessarily uniquely determined by $\o$ and $\overline{\nu}$. The value group, residue field, valuation ring and maximal ideal of a valued field $(L,\nu)$ are denoted by $\nu L$, $L\nu$, $\mathcal{O}_L$ and $\mathcal{M}_L$. The value of an element $l\in L$ is denoted by $\nu l$ and its residue by $l\nu$. Denote a valued field extension as $(L|K,\nu)$ where $L|K$ is an extension of fields, $\nu$ is a valuation on $L$ and $K$ is equipped with the restricted valuation. If $L$ and $K$ are subfields of a larger valued field $(F,\nu)$, then we will also write $(L|K,\nu)$ to mean that they are equipped with the restricted valuations. Given an extension of valued fields $(K(y)|K,\nu)$, we define the set
		\[  \nu(y-K):= \{ \nu(y-a)\mid a\in K \}. \]
		The separable closure of a field $K_1$ in some overfield $K_2$ is denoted by $(K_2|K_1)^\sep$.
		
		\subsection{Ramification Theory}
		We recall some aspects of ramification theory and general valuation theory [cf. \ref{Abh book}, \ref{Endler book}, \ref{Engler book}, \ref{ZS2}]. Set $G:= \Gal(\overline{K}|K)$. The absolute decomposition group of the extension $(\overline{K}|K,\overline{\nu})$ is defined as
		\[  G_d := \{\s \in G \mid \overline{\nu} \circ \s = \overline{\nu} \text{ on } K^\sep \}.\]
		The absolute inertia group is defined as
		\[G_i := \{\s \in G \mid \overline{\nu}(\s a - a) > 0 \text{ for all } a \in \mathcal{O}_{K^\sep} \}.\]
		The absolute ramification group is defined as
		\[ G_r := \{\s \in G \mid \overline{\nu}(\s a - a) > \overline{\nu} a \text{ for all } a \in K^\sep \setminus \{0\} \}. \]
		The corresponding fixed fields in $K^\sep$ will be denoted by $K^d$, $K^i$ and $K^r$ and they are called the \textbf{absolute decomposition field}, \textbf{absolute inertia field} and the \textbf{absolute ramification field} of $(K,\nu)$ respectively. For an arbitrary algebraic extension of valued fields $(L|K,\overline{\nu})$, we have that $L^d = L.K^d$, $L^i = L.K^i$ and $L^r = L.K^r$. The \textbf{henselization} of a valued field $K$ is the smallest henselian field containing $K$, and we can consider it to be the same as the absolute decomposition field $K^d$. The henselization is also denoted by $K^h$. As a consequence, for algebraic extensions $L|K$, we have 
		\[ L^h = L.K^h. \]
		
		\pars We state a simple form of the \textbf{fundamental inequality} here: for every finite extension $(L|K,\overline{\nu})$,
		\[ [L:K] \geq (\overline{\nu}L:\nu K)[L\overline{\nu}:K\nu]. \] 
		The extension $(L|K,\overline{\nu})$ is said to be \textbf{defectless} if we have equality in the above inequality.
		
		\pars The \textbf{characteristic exponent} of $(K,\nu)$, denoted by $p$, is defined to be equal to $\ch K\nu$ if $\ch K\nu>0$, and $1$ otherwise. The \textbf{Lemma of Ostrowski} states that whenever $\nu$ admits a unique extension $\overline{\nu}$ to a finite extension $E$ of $K$, then 
		\[ [E:K] = (\overline{\nu}E:\nu K)[E\overline{\nu}:K\nu]p^n \]
		for some $n\in\NN$. As a consequence, any extension of henselian valued fields of degree coprime to $p$ is a defectless extension.
		
		An algebraic extension $(L|K,\overline{\nu})$ of henselian valued fields is said to be \textbf{tame} if every finite subextension $(E|K,\overline{\nu})$ satisfies the following conditions:
		\sn (TE1) $\ch K\nu$ does not divide $(\overline{\nu}E:\nu K)$,
		\n (TE2) the residue field extension $E\overline{\nu}|K\nu$ is separable,
		\n (TE3) $(E|K,\overline{\nu})$ is defectless.\\
		An algebraic extension $(F|K,\overline{\nu})$ of henselian valued fields is said to be \textbf{purely wild} if $\overline{\nu}F/\nu K$ is a $p$-group and $F\overline{\nu}|K\nu$ is purely inseparable. Equivalently, $F\sect K^r = K$. An extension of valued fields is said to be \textbf{immediate} if there is no value group or residue field extension. In particular, immediate extensions and purely inseparable extensions of henselian valued fields are purely wild extensions.
		
		\subsection{Krasner's Lemma} Take a valued field $(K,\nu)$ and set $G:= \Gal(\overline{K}|K)$. Choose $a\in\overline{K}\setminus K$ which is not purely inseparable over $K$. The \textbf{Krasner constant} of $a$ over $K$ is defined as
		\[ \kras(a,K):= \max\{ \overline{\nu}(\s a-\tau a) \mid \s, \, \tau\in G \text{ and } \s a\neq \tau a \}. \]
		The fact that all extensions of $\nu$ to $\overline{K}$ are conjugate implies that $\kras(a,K)$ is independent of the choice of the extension $\overline{\nu}$. For the same reason, whenever $\nu$ admits a unique extension from $K$ to $K(a)$, we can define the Krasner's constant as 
		\[ \kras(a,K):= \max\{ \overline{\nu}(a-\s a) \mid \s \in G \text{ and } a\neq \s a \}. \]
		We will now state a variant of the important Krasner's Lemma [cf. Lemma 2.21, \ref{Kuh value groups residue fields rational fn fields}]. 
		\begin{Lemma}\label{Lemma Krasner}
			Take a separable-algebraic extension $K(a)|K$ and $(K(a,y)|K,\nu)$ be any valued extension such that 
			\[ \nu(a-y)>\kras(a,K). \]
			Then for every extension of $\nu$ from $K(a)$ to $\overline{K(a,y)}$, we have that $a\in K(y)^h$. 
		\end{Lemma} 
		Take an algebraic extension $(L|K,\overline{\nu})$ and a separable-algebraic element $a$ over $K$. Then $a$ is separable-algebraic over $L$. Further, the minimal polynomial of $a$ over $L$ divides that over $K$. It then follows from the definition of the Krasner constant that
		\[ \kras(a,L)\leq \kras(a,K). \]

		\subsection{Artin-Schreier extensions} Take a valued field $(K,\nu)$ with $\ch K=p>0$ and a polynomial $f(x):=x^p-x-c$ over $K$. Take $a\in\overline{K}$ such that $f(a)=0$. Such a polynomial is said to be an Artin-Schreier polynomial.  When $K(a)|K$ is a non-trivial extension, it is said to be an Artin-Schreier extension. Applying the Frobenius endomorphism, we observe that
		\[ f(a+i) = (a+i)^p - (a+i)-c = a^p+i^p-a-i-c= i^p-i=0 \]
		for all $i\in\FF_p$. Thus the conjugates of $a$ are $a,a+1,\dotsc,a+p-1$. It follows that $f(x)$ is either irreducible, or it splits completely over $K$. Further, whenever $\overline{\nu}a\geq 0$, then $\nu c= \overline{\nu}(a^p-a)\geq 0$. Thus if we have $\nu c<0$, then $\overline{\nu}a<0$. It then follows from the triangle inequality that 
		\[ \nu c<0 \Longrightarrow \nu c = \overline{\nu}a^p = p\overline{\nu}a. \]

		\subsection{Pseudo-Cauchy sequences} A well-ordered set $(z_\mu)_{\mu<\l}$ in a valued field $(K,\nu)$, where $\l$ is some limit ordinal, is said to form a \textbf{pseudo-Cauchy sequence} if $\nu(z_{\mu_1}-z_{\mu_2}) < \nu(z_{\mu_2}-z_{\mu_3})$ for all $\mu_1<\mu_2<\mu_3<\l$. It follows from the triangle inequality that $\nu(z_\mu - z_{\rho}) = \nu(z_\mu-z_{\mu+1})$ for all $\mu<\rho<\l$. An element $l$ in some valued field extension $(L,\nu)$ of $(K,\nu)$ is said to be a \textbf{limit} of $(z_\mu)_{\mu<\l}$ if $\nu(l -z_\mu) = \nu(z_\mu-z_{\mu+1})$ for all $\mu<\l$.
		\newline Take a polynomial $f(x)\in K[x]$ and a pseudo-Cauchy sequence $(z_\mu)_{\mu<\l}$ over $K$. It follows from Lemma 5 of [\ref{Kaplansky}] that the sequence $(\nu f(z_\mu))_{\mu<\l}$ is either ultimately constant, or it is ultimately monotonically increasing. If $(\nu f(z_\mu))_{\mu<\l}$ is ultimately constant for all polynomials $f$ over $K$, then $(z_\mu)_{\mu<\l}$ is said to be a pseudo-Cauchy sequence of \textbf{transcendental type}. Otherwise, it is of \textbf{algebraic type}. For a pseudo-Cauchy sequence $(z_\mu)_{\mu<\l}$ of algebraic type, a monic polynomial $f(x)\in K[x]$ of minimal degree such that $(\nu f(z_\mu))_{\mu<\l}$ is ultimately monotonically increasing is said to be an \textbf{associated minimal polynomial} of $(z_\mu)_{\mu<\l}$.
		
		\begin{Proposition}\label{Proposition Kaplansky}
			Take an extension $(K(y)|K,\nu)$ and a pseudo-Cauchy sequence $(z_\mu)_{\mu<\l}$ over $(K,\nu)$. Then either $y$ is a limit of $(z_\mu)_{\mu<\l}$, or, $(\nu (y-z_\mu))_{\mu<\l}$ is ultimately constant.
		\end{Proposition}
		
		\begin{proof}
			Consider the polynomial $f(x) = x-y \in K(y)(x)$. Observe that $(z_\mu)_{\mu<\l}$ is a pseudo-Cauchy sequence over $(K(y),\nu)$. Assume that $(\nu(y-z_\mu))_{\mu<\l} = (\nu f(z_\mu))_{\mu<\l}$ is not ultimately constant. In light of Lemma 5 of [\ref{Kaplansky}] the sequence $(\nu f(z_\mu))_{\mu<\l}$ is ultimately monotonically increasing. So there exists some ordinal $\mu_0<\l$ such that $\nu (y-z_{\mu_2})>\nu (y-z_{\mu_1})$ for all $\mu_0<\mu_1<\mu_2<\l$. It follows from the triangle inequality that 
			\[ \nu(y-z_\mu) = \nu(z_\mu- z_{\mu +1}) \text{ for all } \mu>\mu_0. \]
			Now take any ordinal $\mu^\prime<\mu_0$ and $\mu>\mu_0$. Then $\nu(y-z_\mu) = \nu(z_\mu- z_{\mu +1})>\nu (z_\mu-z_{\mu^\prime})$. As a consequence of the triangle inequality, we now obtain that
			\[ \nu(y-z_\mu) = \nu(z_\mu- z_{\mu +1}) \text{ for all } \mu<\l. \]
			Hence $y$ is a limit of $(z_\mu)_{\mu<\l}$. 
		\end{proof}

		\subsection{Pure and weakly pure extensions} An extension $(K(x)|K,\o)$ is said to be \textbf{pure} if it satisfies one of the following conditions:
		\sn (PE1) $\o(x-a)$ is not a torsion element modulo $\nu K$ for some $a\in K$,
		\n (PE2) $\o b(x-a)=0$ and $b(x-a)\o$ is transcendental over $K\nu$ for some $a,b \in K$,
		\n (PE3) $x$ is the limit of some pseudo-Cauchy sequence of transcendental type in $(K,\nu)$.\\
		It has been shown in Lemma 3.5 of [\ref{Kuh value groups residue fields rational fn fields}] that $\o K(x)/\nu K$ is torsion free and $K\nu$ is relatively algebraically closed in $K(x)\o$ when $(K(x)|K,\o)$ is a pure extension. As a consequence we obtain the following facts:
		\sn (i) In case (PE1), $\o K(x) = \nu K\dirsum \ZZ\o(x-a) \text{ and } K(x)\o = K\nu$.
		\n (ii) In case (PE2), $\o K(x) =\nu K$ and $K(x)\o = K\nu (b(x-a)\o)$.\\  
		An extension $(K(x)|K,\o)$ is always pure when $K$ is algebraically closed [Lemma 3.6, \ref{Kuh value groups residue fields rational fn fields}]. 
		\pars An extension $(K(x)|K,\o)$ is said to be \textbf{weakly pure} if it is either pure, or there exist $a,b\in K$ and $e\in\NN$ such that $\o b(x-a)^e = 0$ and $b(x-a)^e\o$ is transcendental over $K\nu$.

		\subsection{Homogeneous Sequences} The notion of homogeneous sequences was introduced by Kuhlmann [\ref{Kuh value groups residue fields rational fn fields}, \ref{Kuh corrections value groups, residue fields, bad places}]. An element $a$ is said to be \textbf{strongly homogeneous} over $(K,\nu)$ if $a\in K^\sep\setminus K$, the extension of $\nu$ from $K$ to $K(a)$ is unique and $\overline{\nu}a = \kras(a,K)$. 
		\newline Let $(L|K,\nu)$ be an extension of valued fields and $a, y\in L$. We will say that $a$ is a \textbf{homogeneous approximation} of $y$ over $K$ if $a-d$ is strongly homogeneous over $K$ for some $d\in K$, and $\nu(y-a)> \nu(a-d)$. A sequence of elements $(a_i)_{i\in S}$ in $\overline{K}$, where $S$ is an initial segment of $\NN$ and $a_0:=0$, is said to form a \textbf{homogeneous sequence for $y$} if $a_i - a_{i-1}$ is a homogeneous approximation of $y-a_{i-1}$ over $K(a_0,\dotsc,a_{i-1})$ for all $i\in S$.

		\subsection{Minimal pairs and key polynomials} Take a polynomial $f(x)\in K[x]$ with $\deg f =n$ and take $a\in K$. Write 
		\[ f(x) = \sum_{i=0}^{n} c_i (x-a)^i \]
		where $c_i\in K$. Now take some $\g$ in some ordered abelian group containing $\nu K$. Define the map $\nu_{a,\g}:K[x]\longrightarrow \nu K +\ZZ\g$ by setting 
		\[ \nu_{a,\g} f:= \min \{ \nu c_i +i\g \}. \]
		Extend $\nu_{a,\g}$ canonically to $K(x)$. Then $\nu_{a,\g}$ is a valuation on $K(x)$ [Lemma 3.10, \ref{Kuh value groups residue fields rational fn fields}]. By definition, $\nu_{a,\g}(x-a)=\g$. Take any $b\in K$. Then $\nu_{a,\g}(x-b) = \min\{ \g,\nu(a-b) \} \leq \g$. It follows that 
		\begin{equation}
			\nu_{a,\g}(x-a)=\g=\max\nu_{a,\g}(x-K).
		\end{equation}
		It has been shown in Theorem 3.11 of [\ref{Kuh value groups residue fields rational fn fields}] that $(K(x)|K,\o)$ is a valuation transcendental extension if and only if $\overline{\o} = \overline{\nu}_{a,\g}$ for some $a\in \overline{K}$ and $\g\in\overline{\o}\overline{K}(x)$. Further, the extension is value transcendental if and only if $\g\notin\overline{\nu}\overline{K}$ and the extension is residue transcendental if and only if $\g\in\overline{\nu}\overline{K}$. Such a pair $(a,\g)$ is said to be a \textbf{pair of definition} for $\o$. A valuation transcendental extension can have multiple pairs of definition. The following lemma, proved in Proposition 3 of [\ref{AP sur une classe}], gives us a relation between the different pairs of definition: 
		
		\begin{Lemma}
			Take $a,\,a^\prime \in\overline{K}$ and $\g,\g^\prime$ in some ordered abelian group containing $\overline{\nu}\overline{K}$. Then,
			\[ \overline{\nu}_{a,\g} =\overline{\nu}_{a^\prime,\g^\prime} \text{ if and only if } \g=\g^\prime \text{ and } \overline{\nu}(a-a^\prime)\geq\g. \]
		\end{Lemma}
	
		The above lemma motivates the following definition:
		\begin{Definition}
			A pair $(a,\g)$ is said to be a \textbf{minimal pair of definition} for $\o$ if $(a,\g)$ is a pair of definition for $\o$ and for any $b\in\overline{K}$,
			\[ \overline{\nu}(a-b)\geq\g\Longrightarrow [K(a):K] \leq [K(b):K]. \]
		\end{Definition}
	    In other words, $a$ is of minimal degree over $K$ with the property that $(a,\g)$ is a pair of definition for $\o$. It follows from Theorem 3.11 of [\ref{Kuh value groups residue fields rational fn fields}] that every valuation transcendental extension admits a pair of definition. The well-ordering principle then implies that every valuation transcendental extension admits a minimal pair of definition.

		\pars Take a polynomial $f(x)\in K[x]$. We define
		\[ \d(f) := \max\{ \overline{\o}(x-a) \mid f(a)=0 \}. \]
		This notion has been introduced by Novacoski in [\ref{Novacoski key poly and min pairs}]. It has been observed in [\ref{Novacoski key poly and min pairs}] that for a fixed valuation transcendental extension $\o$, the value $\d(f)$ is independent of the extension $\overline{\o}$ of $\o$ to $\overline{K}(x)$. A root $a$ of $f$ such that $\d(f) = \overline{\o}(x-a)$ will be referred to as a \textbf{maximal root} of $f$. A monic polynomial $Q(x)\in K[x]$ is said to be a \textbf{key polynomial} of $\o$ over $K$ if for any $f(x)\in K[x]$, we have that
		\[ \deg f < \deg Q \Longrightarrow \d(f) < \d(Q). \]
		Now take a monic polynomial $Q(x)\in K[x]$. Then any polynomial $f(x) \in K[x]$ has a unique expansion, called the \textbf{$Q$-expansion}, which is of the form 
		\begin{equation}
			f = f_0 + f_1 Q + \dotsc + f_r Q^r,
		\end{equation}	
		where $f_i(x) \in K[x]$ such that $\deg f_i < \deg Q$. We have the map $\nu_Q : K[x] \longrightarrow \o K(x)$ given by
		\[ \nu_Q f:= \min\{ \o f_i + i \o Q \}. \]	
		We can extend $\nu_Q$ by setting $\nu_Q (f/g) := \nu_Q f - \nu_Q g$. A sufficient condition for $\nu_Q$ to be a valuation on $K(x)$ is that $Q$ be a key polynomial of $\o$ over $K$ [Proposition 2.6, \ref{Nova Spiva key pol pseudo convergent}], but it is not a necessary condition [Proposition 2.3, \ref{Novacoski key poly and min pairs}]. The connection between minimal pairs of definition and key polynomials is explored in Theorem 1.1 of [\ref{Novacoski key poly and min pairs}]: 
		
		\begin{Theorem}
			Take a monic irreducible polynomial $Q(x)\in K[x]$. Take a root $a$ of $Q(x)$ such that $\d(Q)=\overline{\o}(x-a)$. Then $(a, \d(Q))$ is a minimal pair of definition for $\o$ if and only if $\o = \nu_Q$.
		\end{Theorem}	 
		
		\begin{Corollary}\label{corollary to Novacoski}
			Take a minimal pair of definition $(a,\g)$ for $\o$ and consider the minimal polynomial $Q(x)$ of $a$ over $K$. Then $\o=\nu_Q$. Further, $Q$ is a key polynomial for $\o$.
		\end{Corollary}
		
		\begin{proof}
			By definition, we can take an extension $\overline{\o}$ of $\o$ to $\overline{K}(x)$ which has $(a,\g)$ as a minimal pair of definition, that is, $\overline{\o} = \overline{\nu}_{a,\g}$. It follows that
			\[ \overline{\o}(x-a)=\g=\max\overline{\o}(x-\overline{K}). \]
			The maximality of $\g$ then implies that $\d(Q)=\g$. It now follows from Theorem 1.1 of [\ref{Novacoski key poly and min pairs}] that $\o=\nu_Q$. Take a polynomial $f(x)\in K[x]$ such that $\deg f<\deg Q$. The maximality of $\g$ implies that $\d(f)\leq \d(Q)$. Suppose that $\d(f)=\d(Q)$. Then there exists some root $z$ of $f$ such that $\overline{\o}(x-z)=\g$. However, this implies that $(z,\g)$ is also a pair of definition for $\o$, which contradicts the minimality of $(a,\g)$. Hence $\d(f)<\d(Q)$ and thus $Q(x)$ is a key polynomial for $\o$.
		\end{proof}
	
	
\section{Minimal pairs of definition for $\o$} \label{section min pair}	
	\underline{Throughout the rest of the paper}, $\o$ denotes some valuation transcendental extension of $\nu$ to $K(x)$, $\overline{\nu}$ denotes some extension of $\nu$ to a fixed algebraic closure $\overline{K}$ of $K$, and $\overline{\o}$ denotes some common extension of $\o$ and $\overline{\nu}$ to $\overline{K}(x)$.
	
	\pars Minimal pairs of definition were used by Alexandru, Popescu and Zaharescu [\ref{AP sur une classe}, \ref{APZ characterization of residual trans extns}, \ref{APZ2 minimal pairs}] to give a very satisfactory description of residue transcendental extensions. We first provide a brief summary of some of their important results. A proof of these statements can be obtained in Thereom 2.1 of [\ref{APZ characterization of residual trans extns}].
	
	\begin{Remark}\label{Remark r.t. value grp res field}
		Assume that $\o$ is a residue transcendental extension of $\nu$ to $K(x)$ with a minimal pair of definition $(a,\g)$. Take the minimal polynomial $Q(x)$ of $a$ over $K$. Take a polynomial $g(x)\in K[x]$ such that $\deg g < \deg Q$. Then 
		\[ \o g=\overline{\nu}g(a). \]
		Take a polynomial $f(x) \in K[x]$ with the $Q$-expansion $f = f_0 + f_1 Q + \dotsc + f_r Q^r$. Then, 
		\[ \o f = \min \{ \overline{\nu}f_i(a) + i \o Q \}. \]
		We have that
		\[ \o K(x) = \overline{\nu}K(a) + \ZZ\o Q  \text{ and } (\o K(x):\overline{\nu}K(a)) = e, \]
		where $e$ is the least natural number such that $e\o Q\in \overline{\nu}K(a)$.
		
		If $h_1,h_2$ are coprime polynomials over $K$, then we define the \textbf{order} of $\frac{h_1}{h_2}$ as 
		\[ \ord(\frac{h_1}{h_2}) := \max\{\deg h_1,\deg h_2\}. \]    
		Now take $g(x)\in K[x]$ such that $\deg g< \deg Q$ and $\overline{\nu}g(a) = \o g = \o Q^e$. Then $\frac{Q^e}{g}$ is the element of $\mathcal{O}_{K(x)}$ of the smallest order such that $\frac{Q^e}{g}\o$ is transcendental over $K\nu$. Further, $K(a)\overline{\nu}$ can be canonically identified with the relative algebraic closure of $K\nu$ in $K(x)\o$. Moreover,
		\[ K(x)\o = K(a)\overline{\nu}(\frac{Q^e}{g}\o). \]
		Now take some $h(x)\in K[x]$ such that $\deg h < \deg Q$ and $\o h = -e\o Q$. Then $\o Q^eh = 0$ and hence $(Q^eh)\o \in K(x)\o$. Further, observe that $\o (gh)=0$ and hence 
		\[ (hQ^e)\o = ((gh)\o)(\frac{Q^e}{g}\o). \]
		Suppose that $(gh)\o$ is transcendental over $K\nu$. It follows from Proposition 1.1 of [\ref{APZ2 minimal pairs}] that there is a root $b$ of $gh\in K[x]$ such that $(b,\g)$ is a pair of definition for $\o$. So either $g(b) =0$ or $h(b)=0$. However, the fact that $\deg g, \deg h < \deg Q$ and the minimality of $(a,\g)$ yields a contradiction to the above assertion. Hence $(gh)\o$ is algebraic over $K\nu$. The fact that $K(a)\overline{\nu}$ is the relative algebraic closure of $K\nu$ in $K(x)\o$ now implies that $(gh)\o\in K(a)\overline{\nu}$. It follows that $K(a)\overline{\nu}(\frac{Q^e}{g}\o) = K(a)\overline{\nu}((hQ^e)\o)$. We have thus obtained that
		\[ K(a)\overline{\nu}(\frac{Q^e}{g}\o) = K(x)\o = K(a)\overline{\nu}((hQ^e)\o). \] 
	\end{Remark}
	
	We now show that a value transcendental extension is completely described by a minimal pair of definition. Recall that a minimal pair of definition always exists for a valuation transcendental extension. 

\begin{Lemma}\label{Lemma value transcendental extension is described by min pair}
	Assume that $\o$ is value transcendental. Take a minimal pair of definition $(a,\g)$ for $\o$ and the minimal polynomial $Q(x)$ of $a$ over $K$. Then $\o Q \notin \overline{\nu} \overline{K}$. Take a polynomial $g(x)\in K[x]$ such that $\deg g < \deg Q$. Then $\o g=\overline{\nu}g(a)$. Take a polynomial $f(x) \in K[x]$ with the $Q$-expansion $f = f_0 + f_1 Q + \dotsc + f_r Q^r$. Then, 
	\[ \o f = \min \{ \overline{\nu}f_i(a) + i \o Q \}. \]
\end{Lemma}	
	
\begin{proof}
	It follows from Corollary \ref{corollary to Novacoski} that $\o = \nu_Q$. So for a polynomial $f(x) \in K[x]$ with the given $Q$-expansion, we have
	\begin{equation}\label{eqn omega f in Q expansion}
		\o f = \min\{ \o f_i + i\o Q \}.
	\end{equation}
	Take $g(x) \in K[x]$ with $\deg g < \deg Q$. Write $g(x) = c(x-b_1) \dotsc (x-b_m)$ where $b_i \in \overline{K}$ and $c\in K$. If $\overline{\nu}(a-b_i) > \g$ for some $i$, then $(b_i, \g)$ is a pair of definition for $\overline{\o}$. The minimality of $(a,\g)$ then implies that $[K(b_i):K] \geq [K(a):K]$. But, $[K(b_i):K] \leq \deg g < \deg Q = [K(a):K]$. Hence $\overline{\nu}(a-b_i) < \g$ for each $i$, and thus, $\overline{\o}(x-b_i) = \overline{\nu}(a-b_i) \in \overline{\nu}\overline{K}$. It follows that, 
	\begin{equation}\label{eqn deg g < deg Q}
		\deg g < \deg Q \Longrightarrow \o g = \overline{\nu} g(a)\in \overline{\nu}\overline{K}.
	\end{equation}
	From (\ref{eqn omega f in Q expansion}) and (\ref{eqn deg g < deg Q}) it follows that 
	\[ \o f = \min \{ \overline{\nu}f_i(a) + i \o Q \}. \]
	Finally, we express $Q(x) = (x-a)(x-a_2) \dotsc (x-a_n)$. We have $\overline{\o}(x-a) = \g\notin\overline{\nu}\overline{K}$. Further, for each $i$ we have that $\overline{\o}(x-a_i) = \min\{\g, \overline{\nu}(a-a_i) \}$. Hence $\o Q \notin \overline{\nu}\overline{K}$.
\end{proof}	

\begin{Remark}\label{Remark value group res field}
	Take notations and assumptions as in Lemma \ref{Lemma value transcendental extension is described by min pair}. Then $\o K(x) = \overline{\nu}K(a)\dirsum\ZZ\o Q$ and $K(x)\o = K(a)\overline{\nu}$. The equality of the value groups is straightforward. To see the equality of the residue fields, we first observe that $(K(a,x)|K(a),\overline{\o})$ is a pure value transcendental extension. Consequently, $K(a,x)\overline{\o} = K(a)\overline{\nu}$. It follows that
	\[  K(x)\o\subseteq K(a,x)\overline{\o} = K(a)\overline{\nu}. \]
	Now take $g(x)\in K[x]$ with $\deg g < \deg Q$ and $\o g = 0$. Consider the expansion 
	\[ g(x) = g(a) + \sum_{i=1}^{m} c_i (x-a)^i. \]
	Observe that all the monomials in the above expansion have distinct values. It now follows from Lemma \ref{Lemma value transcendental extension is described by min pair} and the triangle inequality that
	\[ \o g=0=\overline{\nu}g(a) < \overline{\o} [c_i (x-a)^i] \text{ for all }i. \] 
	Consequently,
	\[ \overline{\o}(g(x)-g(a)) = \min\{ \overline{\o} [c_i(x-a)^i] \mid i=1,\dotsc,m \} >0. \]
	We have thus obtained that $g(x)\o = g(a)\overline{\nu}$. Since this holds for any arbitrary $g(a)\overline{\nu}\in K(a)\overline{\nu}$, we have
	\[ K(a)\overline{\nu}\subseteq K(x)\o. \]
	Hence we have arrived at the equality
	\[ K(x)\o = K(a)\overline{\nu}. \]
\end{Remark}

In Theorem 3.1 of [\ref{APZ2 minimal pairs}] it is shown that every residue transcendental extension of $\nu$ to $K(x)$ has a minimal pair of definition $(a,\g)\in\overline{K}\times\overline{\nu}\overline{K}$ where $a$ is separable over $K$. However, this is not necessarily true for a value transcendental extension.

\begin{Lemma}
	Assume that $(a,\g)$ is a minimal pair of definition for $\o$ and $\g> \overline{\nu}\overline{K}$. Then $\o$ is value transcendental and $(a,\g)$ is the unique pair of definition for $\o$. 
\end{Lemma}

\begin{proof}
	The fact that $\g$ is not contained in the divisible group $\overline{\nu}\overline{K}$ implies that $\g$ is torsion free modulo $\nu K$. Consequently, $\o$ is a value transcendental extension. Suppose that $(b,\g)$ is another pair of definition for $\o$ for some $b\in\overline{K}$ with $b\neq a$. Then $\overline{\nu}(a-b)>\g$, which contradicts our assumption. Hence $(a,\g)$ is the unique pair of definition for $\o$.
\end{proof}

\begin{Corollary}
	Assume that $K$ is not perfect. Then there exists a value transcendental extension $\o$ of $\nu$ to $K(x)$ with a unique pair of definition $(a,\g)$ such that $a$ is not separable over $K$.
\end{Corollary}

\begin{proof}
	Consider the ordered abelian group $\G:= \ZZ\dirsum \overline{\nu}\overline{K}$ equipped with the lexicographic order. Take some $a\in \overline{K}\setminus K^\sep$ and set $\g:= (1,0)$. Take the extension $\overline{\o}:= \overline{\nu}_{a,\g}$ of $\overline{\nu}$ to $\overline{K}(x)$ and set $\o:= \overline{\o}|_{K(x)}$. Observe that $\g>\overline{\nu}\overline{K}$. The result now follows from the preceding lemma.
\end{proof}

\begin{Proposition}
	Assume that $\nu K$ is cofinal in $\o K(x)$. Then $\o$ has a minimal pair of definition $(b,\g)$ such that $b$ is separable over $K$.
\end{Proposition}	

\begin{proof}
	Observe that $\overline{\o}\overline{K}(x)/\o K(x)$ is a torsion group and hence $\o K(x)$ is cofinal in $\overline{\o}\overline{K}(x)$. Then the fact that $\nu K$ is cofinal in $\o K(x)$ implies that $\nu K$ is cofinal in $\overline{\o}\overline{K}(x)$. The proof is now identical to the proof of Theorem 3.1 of [\ref{APZ2 minimal pairs}].
\end{proof}	

We combine the preceding observations in the next theorem. 

\begin{Theorem}\label{Thm min pair of definition with sep a}
	Take a minimal pair of definition $(a,\g)$ for $\o$. Then the following cases are possible:
	\sn (i) $\g > \overline{\nu}\overline{K}$. Then $\o$ is value transcendental and $(a,\g)$ is the unique pair of definition for $\o$.
	\n (ii) $\nu K$ is cofinal in $\o K(x)$. Then there exists a minimal pair of definition $(b,\g)$ for $\o$ such that $b$ is separable over $K$.
\end{Theorem}	

We now consider the following question: given a pair of definition $(a,\g)$ for $\o$, can we construct a minimal pair of definition for $\o$? We illustrate a solution in the particular case when $K(a)|K$ is a defectless extension of henselian valued fields. Assume that $(K,\nu)$ is henselian and take $a\in\overline{K}$. Consider the set
\[ M(a):= \{ \overline{\nu}(a-z)\mid z\in\overline{K} \text{ and } [K(z):K] < [K(a):K] \}. \]
The maximum of $M(a)$, whenever it exists, is denoted by $\d(a,K)$. It has been shown in Theorem 1.3 of [\ref{AB SKK alg max defectless}] that $M(a)$ admits a maximum whenever $(K(a)|K,\overline{\nu})$ is a defectless extension. A pair $(a,z)\in\overline{K}\times\overline{K}$ is said to form a \textbf{distinguished pair} if $[K(z):K] < [K(a):K]$, $\overline{\nu}(a-z) = \d(a,K)$ and $[K(z):K]$ is minimal with respect to this property. A chain $a:= a_0, a_1, \dotsc, a_n$ of elements in $\overline{K}$ is said to form a \textbf{complete distinguished chain} if $(a_i, a_{i+1})$ is a distinguished pair for each $i=0, \dotsc, n-1$ and $a_n\in K$. Observe that in this case we have $\d(a,K)>\d(a_1,K)>\dotsc > \d(a_n,K) = \nu a_n$. The following is Theorem 1.2 of [\ref{KA SKK chains associated with elts henselian}]: 

\begin{Theorem}\label{thm complete dist chain}
	Take a henselian valued field $(K,\nu)$. Then an element $a\in\overline{K}$ has a complete distinguished chain if and only if $(K(a)|K,\overline{\nu})$ is a defectless extension. 
\end{Theorem}

We now take a henselian valued field $(K,\nu)$ and a pair of definition $(a,\g)$ for $\o$. When $\g> \overline{\nu}\overline{K}$ then $(a,\g)$ is the unique pair of definition for $\o$. So we restrict our attention to the case when $\overline{\nu}\overline{K}$ is cofinal in $\overline{\o}\overline{K}(x)$. Assume that $(K(a)|K,\overline{\nu})$ is a defectless extension. Then we have a complete distinguised chain $a=a_0, a_1, \dotsc, a_n$. We first consider the case when $\g> \d(a,K)$. If $(b,\g)$ is another pair of definition for $\o$, then $\overline{\nu}(a-b)\geq\g>\d(a,K)$ and hence $[K(b):K]\geq [K(a):K]$. It follows that $(a,\g)$ is a minimal pair of definition for $\o$. Now assume that $\g\leq \d(a,K)$. Then $\overline{\nu}(a-a_1) = \d(a,K)\geq\g$ and hence $(a_1,\g)$ is a pair of definition for $\o$. If $\g>\d(a_1,K)$ then $(a_1,\g)$ is a minimal pair of definition from our preceding discussions. Otherwise $\g\leq\d(a_1,K)$. If $\g\leq \d(a_i,K)$ for all $i=0,\dotsc,n-1$, then $\overline{\nu}(a_n-a_{n-1})=\d(a_{n-1},K) \geq \g$ and hence $(a_n,\g)$ is a pair of definition for $\o$. The fact that $a_n\in K$ implies that $(a_n,\g)$ is a minimal pair of definition for $\o$. Otherwise, we can choose some $i\in\{0,\dotsc,n-1\}$ such that $\g>\d(a_i,K)$ and $i$ is minimal with respect to this property. We conclude from our preceding arguments that $(a_i,\g)$ is a minimal pair of definition for $\o$. We have thus obtained the following result: 

\begin{Theorem}\label{Thm min pair comp dist chain}
	Take a pair of definition $(a,\g)$ for $\o$. Assume that $(K,\nu)$ is henselian and $(K(a)|K,\overline{\nu})$ is a defectless extension. Take a complete distinguished chain $a=a_0,a_1, \dotsc, a_n$ of $a$. Then $(a_i,\g)$ is a minimal pair of definition for $\o$ for some $i\in\{0,\dotsc,n\}$.
\end{Theorem}

	We now show that two simultaneous extensions of $\o$ and $\overline{\nu}$ to $\overline{K}(x)$ are closely related. 
	
	\begin{Lemma}\label{lemma extns of o}
		Let $\overline{\o}$ and $\overline{\o^\prime}$ be two extensions of $\o$ to $\overline{K}(x)$ with minimal pairs of definition $(a,\g)$ and $(a^\prime,\g^\prime)$ respectively. Take the minimal polynomials $Q(x)$ and $Q^\prime(x)$ of $a$ and $a^\prime$ over $K$. Then $\g=\g^\prime$, $\deg Q = \deg Q^\prime$ and $\o Q= \o Q^\prime$.
	\end{Lemma}
	
	\begin{proof}
		From Corollary \ref{corollary to Novacoski} we have that $\nu_Q = \o = \nu_{Q^\prime}$ and $Q$ and $Q^\prime$ are key polynomials for $\o$. Proposition 2.10 of [\ref{Nova Spiva key pol pseudo convergent}] gives us that 
		\[ \deg Q < \deg Q^\prime \Longrightarrow \d(Q) < \d(Q^\prime) \Longrightarrow \nu_Q Q^\prime < \o Q^\prime. \]
		Then the fact that $\o = \nu_Q$ implies $\deg Q \geq \deg Q^\prime$. Similarly we obtain $\deg Q^\prime \geq \deg Q$ and hence $\deg Q = \deg Q^\prime$. We can further observe that $\d(Q) = \d(Q^\prime)$, that is, $\g=\g^\prime$. Again from Proposition 2.10 $(iii)$ of [\ref{Nova Spiva key pol pseudo convergent}] we have, 
		\[ \o Q < \o Q^\prime \Longleftrightarrow \nu_Q Q^\prime < \o Q^\prime \Longleftrightarrow \d(Q) < \d(Q^\prime). \]
		It follows that $\o Q = \o Q^\prime$.
	\end{proof}
	
	The following is Theorem 2.2 of [\ref{APZ2 minimal pairs}]: 
	
	\begin{Theorem}\label{thm r.t. extns coincide on K(x)}
		Let $\overline{\o}$ and $\overline{\o^\prime}$ be two residue transcendental extensions of $\overline{\nu}$ to $\overline{K}(x)$. Then the following statements are equivalent:
		\sn (i) $\overline{\o}|_{K(x)} = \overline{\o^\prime}|_{K(x)}$,
		\n (ii) there exist minimal pairs of definition $(a,\g)$ and $(a^\prime,\g^\prime)$ of $\overline{\o}$ and $\overline{\o^\prime}$ such that the following conditions are fulfilled:
		\sn (a) $\g = \g^\prime$ and $a, a^\prime$ are conjugates over $K$,
		\n (b) if $g(x)\in K[x]$ is such that $\deg g < [K(a):K]$, then $\overline{\nu} g(a) = \overline{\nu} g(a^\prime)$.
	\end{Theorem}
	In the next theorem, we show that analogous results hold in the case of value transcendental extensions as well.

\begin{Theorem}\label{thm extensions coincide on K(x)}
	Let $\overline{\o}$ and $\overline{\o^\prime}$ be two value transcendental extensions of $\overline{\nu}$ to $\overline{K}(x)$. Then the following are equivalent:
	\sn (i) $\overline{\o}|_{K(x)} = \overline{\o^\prime}|_{K(x)}$,
	\n (ii) there exist minimal pairs of definition $(a,\g)$ and $(a^\prime,\g^\prime)$ for $\overline{\o}$ and $\overline{\o^\prime}$ respectively such that $a$ and $a^\prime$ are conjugates over $K$.\\
	In this case we have that $\g = \g^\prime$. Further, for any polynomial $g(x)\in K[x]$ with $\deg g<[K(a):K]$, we have that $\overline{\nu}g(a)=\overline{\nu}g(a^\prime)$.   
\end{Theorem}

\begin{proof}
	We first prove the forward direction of the equivalence. Set $\o: = \overline{\o}|_{K(x)} = \overline{\o^\prime}|_{K(x)}$. Take minimal pairs of definition $(a,\g)$ and $(b,\g^\prime)$ for $\overline{\o}$ and $\overline{\o^\prime}$ respectively, and consider the minimal polynomials $Q$ and $Q^\prime$ of $a$ and $b$ over $K$. It follows from Lemma \ref{lemma extns of o} that $\g=\g^\prime$ and $\deg Q = \deg Q^\prime$. The fact that $\d(Q)$ is independent of the extension of $\o$ implies that $\overline{\o^\prime}(x-a^\prime)=\g$ for some conjugate $a^\prime$ of $a$. Hence $(a^\prime,\g)$ is a pair of definition for $\overline{\o^\prime}$. The fact that $\deg Q=\deg Q^\prime$ then implies that $(a^\prime,\g)$ is a minimal pair of definition for $\overline{\o^\prime}$. 
	
	\pars We now prove the reverse direction. Take minimal pairs of definition $(a,\g)$ and $(a^\prime,\g^\prime)$ for $\overline{\o}$ and $\overline{\o^\prime}$ respectively such that $a$ and $a^\prime$ are conjugates over $K$. Take the minimal polynomial $Q(x)$ of $a$ over $K$. Then $Q(x)$ is also the minimal polynomial of $a^\prime$ over $K$. It now follows from Corollary \ref{corollary to Novacoski} that $\overline{\o}|_{K(x)} = \nu_Q = \overline{\o^\prime}|_{K(x)}$.
	
	\pars We now assume that $\overline{\o}$ and $\overline{\o^\prime}$ coincide on $K(x)$ and we denote the restricted valuation by $\o$. It follows from Lemma \ref{Lemma value transcendental extension is described by min pair} that if $g(x)\in K[x]$ is a polynomial such that $\deg g< [K(a):K]$, then $\overline{\nu}g(a)=\o g = \overline{\nu}g(a^\prime)$.
\end{proof}	
	
\begin{Corollary}
	There exist finitely many simultaneous extensions of $\o$ and $\overline{\nu}$ to $\overline{K}(x)$.
\end{Corollary}	
	
\begin{proof}
	Let $\overline{\o}$ be a simultaneous extension of $\o$ and $\overline{\nu}$ and let $(a,\g)$ be a minimal pair of definition for it. It follows from Theorem \ref{thm r.t. extns coincide on K(x)} and Theorem \ref{thm extensions coincide on K(x)} that any other simultaneous extension has a minimal pair of definition of the form $(a^\prime,\g)$, where $a^\prime$ is a conjugate of $a$. Hence there are at most $[K(a):K]$ distinct simultaneous extensions of $\o$ and $\overline{\nu}$ to $\overline{K}(x)$.
\end{proof}	
	
\begin{Remark}
	Take a minimal pair of deinition $(a,\g)$ for $\o$. Following the notation of [\ref{APZ2 minimal pairs}], we define:
	\[ [K:\o]:= [K(a):K]. \]
	Take an extension $\overline{\o}$ of $\o$ to $\overline{K}(x)$ such that $(a,\g)$ is a minimal pair of definition for $\overline{\o}$. We observe that $[K:\o]$ does not depend on the choice of the minimal pair of definition for $\overline{\o}$. It follows from Lemma \ref{lemma extns of o} that $[K:\o]$ is also independent of the extension $\overline{\o}$ and the valuation $\overline{\nu}$. Hence $[K:\o]$ depends only on $K$ and the value transcendental extension $\o$. 
\end{Remark}


\section{Minimal fields of definition for $\o$}\label{section min fields}

\begin{Definition}
		Take a minimal pair of definition $(a,\g)$ for a valuation transcendental extension $\o$ of $\nu$ to $K(x)$. The field $K(a)$ is then said to be a \textbf{minimal field of definition for $\o$}.
\end{Definition}
\underline{Throughout the rest of the paper}, given a minimal field of definition $K(a)$ for $\o$, we will fix an extension $\overline{\o}$ which has $(a,\g)$ as a minimal pair of definition for some $\g\in\overline{\o}\overline{K}(x)$. Further, we will fix an extension of $\overline{\o}$ to $\overline{K(x)}$ and denote it again by $\overline{\o}$.
	
We observe that the minimal fields of definition share some common ramification theoretic properties.

\begin{Proposition}
Take two extensions $\overline{\o}$ and $\overline{\o^\prime}$ of $\o$ to $\overline{K}(x)$ with minimal pairs of definition $(a,\g)$ and $(a^\prime,\g)$ respectively. Then $(\overline{\o}K(a):\nu K) = (\overline{\o^\prime}K(a^\prime):\nu K)$ and $K(a)\overline{\o}=K(a^\prime)\overline{\o^\prime}$.
\end{Proposition}
	
	\begin{proof}
	Take the minimal polynomials $Q(x)\in K[x]$ and $Q^\prime(x)\in K[x]$ of $a$ and $a^\prime$ over $K$. It follows from Corollary \ref{corollary to Novacoski} that $\nu_Q = \o = \nu_{Q^\prime}$. We also obtain from Lemma \ref{lemma extns of o} that $\o Q = \o Q^\prime$ and $\deg Q = \deg Q^\prime$. 
	
	Assume that $\o$ is value transcendental. We directly observe from Remark \ref{Remark value group res field} that $K(a)\overline{\o} = K(x)\o = K(a^\prime)\overline{\o^\prime}$. Further, $\overline{\o}K(a)\dirsum\ZZ\o Q = \o K(x) = \overline{\o^\prime}K(a^\prime)\dirsum\ZZ\o Q $ as $\o Q = \o Q^\prime$. It follows that $\overline{\o}K(a) = \overline{\o^\prime}K(a^\prime)$.
	
	We now assume that $\o$ is residue transcendental. We observe from Remark \ref{Remark r.t. value grp res field} that $K(a)\overline{\o}$ is the relative algebraic closure of $K\nu$ in $K(x)\o$, and the same observation also holds for $K(a^\prime)\overline{\o^\prime}$. Hence $K(a)\overline{\o} = K(a^\prime)\overline{\o^\prime}$. Take the least positive integers $e$ and $e^\prime$ such that $e\o Q\in\overline{\o}K(a)$ and $e^\prime\o Q^\prime\in\overline{\o^\prime}K(a^\prime)$. Take polynomials $g,g^\prime$ in $K[x]$ such that $\deg g, \deg g^\prime < \deg Q$ and $\o \frac{Q^e}{g} = 0 = \o \frac{Q^{\prime e^\prime}}{g^\prime}$. Then it follows from Remark \ref{Remark r.t. value grp res field} that $\frac{Q^e}{g}$ and $\frac{Q^{\prime e^\prime}}{g^\prime}$ are the elements of smallest order in $\mathcal{O}_{K(x)}$ such that their residues are transcendental over $K\nu$. Consequently,
	\[ e\deg Q = \ord (\frac{Q^e}{g}) = \ord (\frac{Q^{\prime e^\prime}}{g^\prime}) = e^\prime \deg Q^\prime \Longrightarrow e=e^\prime, \]
	since $\deg Q = \deg Q^\prime$. It further follows from Remark \ref{Remark r.t. value grp res field} that 
	\[ e(\overline{\o}K(a):\nu K) = (\o K(x):\nu K) = e(\overline{\o^\prime}K(a^\prime):\nu K). \]
	As a consequence we obtain that $(\overline{\o}K(a):\nu K) =  (\overline{\o^\prime}K(a^\prime):\nu K)$. 	
	\end{proof}
	
	The following corollary is now immediate:
	
	\begin{Corollary}
		A minimal field of definition for $\o$ is an immediate extension of $(K,\nu)$ if and only if every minimal field of definition for $\o$ is an immediate extension of $(K,\nu)$. Further, assume that $(K,\nu)$ is henselian. Then a minimal field of definition for $\o$ is a defectless (resp. tame, purely wild) extension of $(K,\nu)$ if and only if every minimal field of definition for $\o$ is a defectless (resp. tame, purely wild) extension of $(K,\nu)$.
	\end{Corollary}


\section{Proof of Theorem \ref{Thm IC and min fields}}\label{section proof of thm 1}

Recall that
\[ IC(K(x)|K,\o) := \overline{K}\sect K(x)^h. \]
The fact that $K(x)^h$ is separable over $K$ implies that $IC(K(x)|K,\o)$ is a separable-algebraic extension of $K$.

\begin{Lemma}\label{Lemma K(a,x)|K(a) is weakly pure}
	Take a pair of definition $(a,\g)$ for $\o$. Then $(K(a,x)|K(a),\overline{\o})$ is a weakly pure extension and $IC(K(a,x)|K(a),\overline{\o}) = K(a)^h$. 
\end{Lemma}

\begin{proof}
	We first consider that $\o$ is value transcendental. Then $\overline{\o}(x-a)$ is not a torsion element modulo $\overline{\nu}K(a)$ and hence $(K(a,x)|K(a),\overline{\o})$ is a pure extension.
	
	\pars We now assume that $\o$ is residue transcendental. Choose the smallest positive integer $e$ such that $e\g\in\overline{\nu}K(a)$. Take $d\in K(a)$ such that $\overline{\nu}d=-e\g$, that is, $\overline{\o}d(x-a)^e = 0$. Suppose that $d(x-a)^e\overline{\o}$ is algebraic over $K(a)\overline{\nu}$. Then there exist $c_i\in\mathcal{O}_{K(a)}$ such that 
	\[ (d(x-a)^e\overline{\o})^n + (c_{n-1}\overline{\nu})(d(x-a)^e\overline{\o})^{n-1}+\dotsc + (c_1\overline{\nu}) d(x-a)^e\overline{\o} + c_0\overline{\nu} =0. \]
	It follows that 
	\[ \overline{\o} \mathlarger{[}(d(x-a)^e)^n + c_{n-1} (d(x-a)^e)^{n-1} +\dotsc + c_1 d(x-a)^e + c_0  \mathlarger{]} > 0.  \]
	The fact that $(a,\g)$ is a pair of definition for $\o$ implies that $\overline{\o} = \overline{\nu}_{a,\g}$. Let $g(x)$ be the polynomial given by $g(x):= (d(x-a)^e)^n + \sum_{i=0}^{n-1} c_i (d(x-a)^e)^i $. By definition, 
	\[ \overline{\o}g = \min \{ \overline{\o}(d(x-a)^e)^n, \overline{\o}[c_i (d(x-a)^e)^i] \mid i=0,\dotsc,n-1 \}. \]
	For any $i\in\{0,\dotsc,n-1\}$, we observe that 
	\[ \overline{\o}[c_i (d(x-a)^e)^i] = \overline{\nu}c_i \geq 0. \] 
	Further, $\overline{\o}(d(x-a)^e)^n = 0$. It follows that $\overline{\o}g=0$ which contradicts our previous observation. Thus $d(x-a)^e\overline{\o}$ is transcendental over $K(a)\overline{\nu}$. Hence $(K(a,x)|K(a),\overline{\o})$ is weakly pure.
	
	\pars It now follows from Lemma 3.7 of [\ref{Kuh value groups residue fields rational fn fields}] that $IC(K(a,x)|K(a),\overline{\o}) = K(a)^h$.
\end{proof}

\pars We now provide a \textbf{proof of Theorem \ref{Thm IC and min fields}}.

\begin{proof}
	 Observe that $IC(K(x)|K,\o) = \overline{K}\sect K(x)^h \subseteq \overline{K}\sect K(a,x)^h = IC(K(a,x)|K(a),\overline{\o})$. It then follows from Lemma \ref{Lemma K(a,x)|K(a) is weakly pure} that
	\begin{equation*}
		IC(K(x)|K,\o) \subseteq K(a)^h.
	\end{equation*}

We first assume that $\o$ is residue transcendental. The fact that $K(b_1)$ is contained in the absolute inertia field of $(K,\nu)$ implies that $K(b_1,x)$ is contained in the absolute inertia field of $(K(x),\o)$ [Theorem 3(2), \ref{Dutta Kuh abhyankars lemma}]. Further, $K(b_1)\overline{\nu}$ is a subfield of $K(a)\overline{\nu}$, which in turn is a subfield of $K(x)\o$ by Remark \ref{Remark r.t. value grp res field}. It now follows from [Theorem 3(2), \ref{Dutta Kuh abhyankars lemma}] that $K(b_1,x)$ is contained in $K(x)^h$. Thus $K(b_1)\subseteq \overline{K}\sect K(x)^h = IC(K(x)|K,\o)$. Since the implicit constant field is henselian, it follows that $K(b_1)^h \subseteq IC(K(x)|K,\o)$.

\pars We now assume that $\o$ is value transcendental. Consider the value transcendental extension $(K(b_2,x)|K(b_2),\overline{\o})$. Take $\g\in\overline{\o}\overline{K}(x)\setminus\overline{\nu}\overline{K}$ such that $(a,\g)$ is a minimal pair of definition for $\o$. Then $(a,\g)$ is a pair of definition for $\overline{\o}$ over $K(b_2)$. Suppose that $(a,\g)$ is not a minimal pair of definition for $\overline{\o}$ over $K(b_2)$. Then there exists some $a^\prime\in\overline{K}$ such that $\overline{\nu}(a-a^\prime)>\g$ and $[K(b_2,a^\prime):K(b_2)] < [K(b_2,a):K(b_2)]$. Consequently, $[K(b_2,a^\prime):K] < [K(b_2,a):K]$. The fact that $b_2\in K(a)$ then implies that $[K(b_2,a^\prime):K] < [K(a):K]$. We further observe that the fact $\overline{\nu}(a-a^\prime)>\g$ implies that $(a^\prime,\g)$ is a pair of definition for $\o$. Then $(a,\g)$ being a minimal pair of definition for $\o$ implies that $[K(a):K]\leq [K(a^\prime):K]$. We thus have the relation:
\[ [K(a^\prime):K] \leq [K(b_2,a^\prime):K] < [K(a):K] \leq [K(a^\prime):K], \]
which is not possible. Hence $(a,\g)$ is a minimal pair of definition for $\overline{\o}$ over $K(b_2)$. Applying the observations of Remark \ref{Remark value group res field} to the extension $(K(b_2,x)|K(b_2),\overline{\o})$, we obtain that $K(b_2,x)\overline{\o} = K(a)\overline{\nu}$. It follows that
\[  K(b_2,x)\overline{\o} = K(x)\o.   \]
The fact that $K(b_2)$ is contained in the absolute ramification field of $(K,\nu)$ implies that $K(b_2,x)$ is contained in the absolute ramification field of $(K(x),\o)$ [Theorem 3(1), \ref{Dutta Kuh abhyankars lemma}]. Further, $\overline{\nu}K(b_2)$ is a subgroup of $\overline{\nu}K(a)$. We thus obtain from Remark \ref{Remark value group res field} that $\overline{\nu}K(b_2)$ is a subgroup of $\o K(x)$. Hence it follows from [Theorem 3(1), \ref{Dutta Kuh abhyankars lemma}] that $K(b_2,x)$ is contained in the absolute inertia field of $(K(x),\o)$. As a consequence we obtain that
\[ \overline{\o}K(b_2,x) = \o K(x). \]
Thus $(K(b_2,x)|K(x),\overline{\o})$ is an immediate extension. The fact that the henselization is an immediate extension implies that $(K(b_2,x)^h|K(x)^h,\overline{\o})$ is also immediate. Now, $(K(b_2,x)^h|K(x)^h,\overline{\o})$ is an immediate subextension of the defectless extension $(K(x)^i|K(x)^h,\overline{\o})$. Hence, $K(x)^h = K(b_2,x)^h$. Thus $b_2\in K(x)^h$ and hence $K(b_2)\subseteq \overline{K}\sect K(x)^h = IC(K(x)|K,\o)$. Since the implicit constant field is henselian, it follows that $K(b_2)^h \subseteq IC(K(x)|K,\o)$. We have thus proved the theorem.
\end{proof}

\begin{Corollary}\label{Corollary a in K^r}
	Take a minimal field of definition $K(a)$ for $\o$. If $\o$ is residue transcendental and $a$ is contained in the absolute inertia field of $(K,\nu)$, then $IC(K(x)|K,\o) = K(a)^h$. If $\o$ is value transcendental and $a$ is contained in the absolute ramification field of $(K,\nu)$, then $IC(K(x)|K,\o) = K(a)^h$.
\end{Corollary}

In this next result, we use the notion of homogenoeus sequences to observe that $IC(K(x)|K,\o) = K(a)$ whenever $(K(a)|K,\overline{\nu})$ is a tame extension of henselian valued fields.  

\begin{Proposition}\label{Prop IC = K(a) when K(a)|K is tame}
	Take a minimal field of definition $K(a)$ for $\o$. Assume that $(K,\nu)$ is henselian and $K(a)$ is contained in the absolute ramification field of $(K,\nu)$. Then $IC(K(x)|K,\o) = K(a)$.
\end{Proposition}

\begin{proof}
	Take $\g\in\overline{\o}\overline{K}(x)$ such that $(a,\g)$ is a minimal pair of definition for $\o$. The fact that $(K(a)|K,\overline{\nu})$ is a tame extension of henselian valued fields implies that there exists a finite homogeneous sequence $(a_1, \dotsc , a_n)$ for $a$ over $(K,\nu)$ such that $a\in K(a_n)$ [Proposition 3.11, \ref{Kuh corrections value groups, residue fields, bad places}]. It follows from Lemma 3.7 of [\ref{Kuh corrections value groups, residue fields, bad places}] that $K(a_1,\dotsc, a_n) = K(a_n) \subseteq K(a)$. Consequently,
	\[ K(a) = K(a_n). \]
	Take any $i<n$. Suppose that $\overline{\nu}(a-a_i) \geq \g$. Then $(a_i,\g)$ is also a pair of definition for $\o$. The minimality of $(a,\g)$ then implies that $[K(a):K] \leq [K(a_i):K]$. It follows from the definition of homogeneous sequences that $a_n \notin K(a_1,\dotsc, a_{n-1})$. Hence we have the relation 
	\[ [K(a):K] \leq [K(a_i):K]< [K(a_n):K] = [K(a):K], \]
	which is not possible. Thus,
	\[ \overline{\nu}(a-a_i) < \g = \overline{\o}(a-x) \text{ for all } i<n. \]
	If $\overline{\nu}(a-a_n) < \g$ then $(a_1, \dotsc , a_n)$ forms a homogeneous sequence for $x$ over $(K,\nu)$ [Lemma 3.6, \ref{Kuh corrections value groups, residue fields, bad places}]. 
	
	Otherwise, assume that $\overline{\nu}(a-a_n)\geq \g$. Then $(a_n,\g)$ is also a minimal pair of definition for $\o$. By definition, $a_i - a_{i-1}$ is a homogeneous approximation of $a-a_{i-1}$ over $(K(a_{i-1}),\overline{\nu})$ for all $i=1,\dotsc, n$. Hence there exists $d_{i-1}\in K(a_{i-1})$ such that $a_i - a_{i-1}-d_{i-1}$ is strongly homogeneous over $(K(a_{i-1}),\overline{\nu})$ and $\overline{\nu}(a-a_i) > \overline{\nu}(a_i - a_{i-1}-d_{i-1})$. Take any $i<n$. Then $\overline{\nu}(a-a_i)<\g = \overline{\o}(a-x)$ and consequently,
	\[ \overline{\o}(x-a_i) = \overline{\nu}(a-a_i)> \overline{\nu}(a_i - a_{i-1}-d_{i-1}). \]
	Thus $a_i - a_{i-1}$ is a homogeneous approximation of $x-a_{i-1}$ over $(K(a_{i-1}),\overline{\nu})$ for all $i<n$. Suppose that $\overline{\nu}(a_n - a_{n-1}-d_{n-1})\geq \g$. The fact that $(a_n,\g)$ is a minimal pair of definition for $\o$ then implies that $(a_{n-1}+d_{n-1},\g)$ is a pair of definition for $\o$ and $[K(a_n):K] \leq [K(a_{n-1}+d_{n-1}):K]$. However, the observation $K(a_{n-1}+d_{n-1}) \subseteq K(a_{n-1})\subsetneq K(a_n)$ now yields a contradiction. It follows that 
	\[ \overline{\o}(x-a_n) = \g > \overline{\nu}(a_n - a_{n-1}-d_{n-1}).   \]
    Thus $a_n - a_{n-1}$ is a homogeneous approximation of $x-a_{n-1}$ over $(K(a_{n-1}),\overline{\nu})$. We have thus obtained that $(a_1,\dotsc, a_n)$ forms a homogeneous sequence for $x$ over $(K,\nu)$.  
	\pars The observation $K(a)=K(a_n)$, coupled with Lemma \ref{Lemma K(a,x)|K(a) is weakly pure}, then implies that $(K(a_n,x)|K(a_n),\overline{\o})$ is weakly pure. The assertion of the proposition now follows from Theorem 3.9 of [\ref{Kuh corrections value groups, residue fields, bad places}].
\end{proof}

\section{Proof of Theorem \ref{Thm gamma and kras}}\label{section proof of thm 2}

\begin{Lemma}\label{Lemma f/g omega in K(a)v}
	Assume that $\o$ is residue transcendental and take a minimal pair of definition $(a,\g)$ for $\o$. Let $f_1(x),\dotsc, f_t(x), g_1(x), \dotsc, g_s(x)$ be polynomials over $K$ and $d\in K(a)$ such that $\deg f_i, \, \deg g_j < [K(a):K]$ for all $i,j$, and $\overline{\o}(df_1\dotsc f_t) = \o(g_1\dotsc g_s)$. Then
	\[ \frac{df_1,\dotsc f_t }{g_1\dotsc g_s}\overline{\o} \in K(a)\overline{\nu}.  \]
\end{Lemma}

\begin{proof}
	Suppose that $\frac{df_1,\dotsc f_t }{g_1\dotsc g_s}\overline{\o}$ is transcendental over $K\nu$. It follows from Proposition 1.1 of [\ref{APZ2 minimal pairs}] that there exists a root $b$ of $df_1\dotsc f_t g_1 \dotsc g_s \in K(a)[x]$ such that $(b,\g)$ is a pair of definition for $\o$. Hence either $f_i (b) =0$ or $g_j (b) =0$ for some $i,j$, and $\overline{\nu}(a-b)\geq \g$. The fact that $\deg f_i, \, \deg g_j < [K(a):K]$, coupled with the minimality of $(a,\g)$ leads to a contradiction. Hence $\frac{df_1,\dotsc f_t }{g_1\dotsc g_s}\overline{\o}$ is algebraic over $K\nu$. We observe that $\frac{df_1,\dotsc f_t }{g_1\dotsc g_s} \in K(a,x)$. It follows from Remark [\ref{Remark r.t. value grp res field}] that $K(a)\overline{\nu}$ is the relative algebraic closure of $K\nu$ in $K(a,x)\overline{\o}$. We have thus proved the lemma. 
\end{proof}

\pars We now provide a \textbf{proof of Theorem \ref{Thm gamma and kras}}.

\begin{proof}
	We first assume that $\g>\kras(a,K)$. It follows from Lemma \ref{Lemma Krasner} that $a\in K(x)^h$. As a consequence we obtain that $K(a)\subseteq IC(K(x)|K,\o)$. The fact that the implicit constant field is henselian, coupled with Theorem \ref{Thm IC and min fields}, gives us that $IC(K(x)|K,\o) = K(a)^h$.
	
	\parm We now assume that $\o$ is value transcendental, $\nu$ admits a unique extension from $K$ to $K(a)$ and $\g<\kras(a,K)$. Take the minimal polynomial $Q(x)$ of $a$ over $K$ and write $Q(x) = (x-a)(x-a_2)\dotsc(x-a_n)$. Assume that $\overline{\nu}(a-a_i)>\g$ for exactly $j$ many conjugates $a_i$ of $a$, including $a$ itself. The fact that $\g<\kras(a,K)$ implies that $j>1$. Further, whenever $\overline{\nu}(a-a_i)<\g$, then $\overline{\o}(x-a_i) = \overline{\nu}(a-a_i)\in\overline{\nu}\overline{K}$. It follows that
	\[ \o Q = j\g + \overline{\nu}c \]
	for some $c\in\overline{K}$. From Remark \ref{Remark value group res field} we now obtain that
	\[ \o K(x) = \overline{\nu}K(a)\dirsum \ZZ(j\g +\overline{\nu}c). \]
	Now, $(K(a,x)|K(a),\overline{\o})$ being a pure value transcendental extension implies that 
	\[ \overline{\o}K(a,x)=\overline{\nu}K(a)\dirsum\ZZ\g. \]
	The fact that $\o K(x)$ is a subgroup of $\overline{\o}K(a,x)$ then implies that $j\g+\overline{\nu}c\in\overline{\nu}K(a)\dirsum\ZZ\g$. As a consequence we obtain that $\overline{\nu}c\in\overline{\nu}K(a)$. Hence, 
	\[ \o K(x) = \overline{\nu}K(a)\dirsum \ZZ j\g. \]
	Thus $(\overline{\o}K(a,x):\o K(x)) =j$. Taking henselizations, we obtain that
	\[ (\overline{\o}K(a,x)^h:\overline{\o} K(x)^h) = j>1. \]  
	Thus $a\notin K(x)^h$ and consequently, $a\notin IC(K(x)|K,\o)$. It now follows from (\ref{eqn IC and min fields}) that 
	\[ IC(K(x)|K,\o) \subsetneq K(a)^h. \]

	\parm We now assume that $\o$ is residue transcendental, $\g\leq \kras(a,K)$ and there is a unique extension of $\nu$ from $K$ to $K(a)$. If $\overline{\o}K(a,x) \neq \o K(x)$, then $K(a,x)^h \neq K(x)^h$ and consequently $a\notin K(x)^h$. It follows from (\ref{eqn r.t. IC and min fields}) that $IC(K(x)|K,\o)\subsetneq K(a)^h$. 
	\pars We now focus our attention to the case when $\overline{\o}K(a,x) = \o K(x)$. Then $(\overline{\o}K(a,x):\overline{\nu}K(a)) = (\o K(x):\overline{\nu}K(a))$. Take the smallest positive integer $e$ such that $e\g\in\overline{\nu}K(a)$. It follows from Remark \ref{Remark r.t. value grp res field} that $(\overline{\o}K(a,x):\overline{\nu}K(a)) = e$. Further, $e$ is also the smallest positive integer such that $e\o Q\in \overline{\nu}K(a)$, where $Q(x)$ is the minimal polynomial of $a$ over $K$.
	\pars Write 
	\[ Q(x) = (x-a)\dotsc(x-a_j)\dotsc(x-a_n), \]
	where $\overline{\nu}(a-a_i) \geq \g$ for $2\leq i\leq j$, and $\overline{\nu}(a-a_i)< \g$ otherwise. The fact that $e\g, e\o Q\in\overline{\nu}K(a)$ implies that $e\overline{\o}((x-a_{j+1})\dotsc(x-a_n)) \in \overline{\nu}K(a)$. Take polynomials $f(x),g(x),h(x)$ over $K$ with $\deg f, \deg g,\deg h < \deg Q$ such that
	\begin{align*}
		\o g &= \overline{\nu}g(a) = -e\g,\\
		\o h&= -e\overline{\o}((x-a_{j+1})\dotsc(x-a_n)),\\
		\o f&= -e\o Q.
	\end{align*}
	It follows that 
	\[ \o f = -e\o Q = -ej\g -e\overline{\o}((x-a_{j+1})\dotsc(x-a_n)) = \o(g^jh). \]
	From Remark \ref{Remark r.t. value grp res field} we obtain that $K(x)\o = K(a)\overline{\nu}(fQ^e\o)$. It follows from Lemma \ref{Lemma f/g omega in K(a)v} that $\frac{g^jh}{f}\o\in K(a)\overline{\nu}$. Thus,
	\[ K(a)\overline{\nu}((g^jhQ^e)\o) = K(a)\overline{\nu}((\frac{g^jh}{f})\o (fQ^e)\o) = K(a)\overline{\nu}((fQ^e)\o). \]
	We have now obtained that
	\begin{equation}\label{eqn K(x) res field}
		K(x)\o  = K(a)\overline{\nu}((g^jhQ^e)\o).
	\end{equation}
	Applying the observations of Remark \ref{Remark r.t. value grp res field} to the residue transcendental extension $(K(a,x)|K(a),\overline{\o})$, we observe that $K(a,x)\overline{\o} = K(a)\overline{\nu}(g(a)(x-a)^e\overline{\o})$. It follows from Lemma \ref{Lemma f/g omega in K(a)v} that $\frac{g(a)}{g}\overline{\o}\in K(a)\overline{\nu}$. Consequently, 
	\[ K(a)\overline{\nu}(g(x-a)^e\overline{\o}) = K(a)\overline{\nu}((\frac{g(a)}{g})\overline{\o}(g(x-a)^e)\overline{\o}) = K(a)\overline{\nu}(g(a)(x-a)^e)\overline{\o}. \]
	We have thus obtained that
	\begin{equation}\label{eqn K(a,x) res field}
		K(a,x)\overline{\o} = K(a)\overline{\nu}(g(x-a)^e\overline{\o}).
	\end{equation}
	Write $Q^e = \sum_{i=e}^{ne} c_i (x-a)^i \text{ where } c_i\in K(a)$. Then
	\[ g^j h Q^e = \sum_{i=e}^{ne} g^j h c_i (x-a)^i. \]
	Now $(a,\g)$ being a pair of definition for $\o$ implies that $\overline{\o} = \overline{\nu}_{a,\g}$. By definition, $\overline{\nu}c_i + i\g \geq e\o Q = -\o (g^jh)$ for all $i$. It follows that $\overline{\o}(g^jhc_i (x-a)^i) \geq 0$ for all $i$. Thus, 
	\[ (g^jhQ^e)\o = \sum_{i=e}^{ne} (g^jhc_i (x-a)^i)\overline{\o}. \]
	Consider some index $i$ which is not a multiple of $e$. Write $i = te +i^\prime$ where $1\leq i^\prime\leq e-1$. Suppose that $\o(g^jh) + \overline{\nu}c_i + i\g =0$. The facts that $c_i \in K(a)$ and $\o(g^j h) = -e\o Q\in\overline{\nu}K(a)$ implies that $i\g\in\overline{\nu}K(a)$. It follows that $i^\prime\g\in\overline{\nu}K(a)$. However, this contradicts the minimality of $e$. Thus, $\overline{\o}(g^jhc_i (x-a)^i) > 0$ whenever $i$ is not a multiple of $e$. Consequently, 
	\[ (g^jhQ^e)\o = \sum_{i=1}^{n} (g^jhc_{ie} (x-a)^{ie})\overline{\o} = \sum_{i=1}^{n} (g^{j-i}hc_{ie} (g(x-a)^e)^i)\overline{\o}. \]
	The fact that $\overline{\o}(g(x-a)^e) =0$ implies that $\overline{\o}(g^{j-i}hc_{ie})\geq 0$. We can thus rewrite the above expression as 
	\[ (g^jhQ^e)\o = \sum_{i=1}^{n} (g^{j-i}hc_{ie})\overline{\o} (g(x-a)^e\overline{\o})^i. \]
	It follows from Lemma \ref{Lemma f/g omega in K(a)v} that whenever $\overline{\o}(g^{j-i}hc_{ie}) =0$, then the residue $(g^{j-i}hc_{ie})\overline{\o}$ is contained in $K(a)\overline{\nu}$. Hence,
	\begin{equation}\label{eqn g^jhQ^e omega in polynomial}
		(g^jhQ^e)\o = \sum_{i=1}^{n} (g^{j-i}hc_{ie})\overline{\o} (g(x-a)^e\overline{\o})^i \in K(a)\overline{\nu} [g(x-a)^e\overline{\o}].
	\end{equation}
	The coefficient of $(g(x-a)^e\overline{\o})^j$ in $(g^jhQ^e)\o$ is given by $(hc_{je})\overline{\o}$. Now, $c_{je}$ is the coefficient of $(x-a)^{je}$ in $Q^e$. Hence,
	\[ c_{je} = (-1)^{ne-je} E_{ne-je} (0,\dotsc,0, a_2-a,\dotsc,a_2-a, \dotsc, a_n-a,\dotsc, a_n-a), \]
	where $0$ and $a_i-a$ appear $e$ times for each $i$, and $E_{ne-je}(x_1,\dotsc,x_{ne})$ is the $(ne-je)$-th elementary symmetric polynomial in the variables $x_1,\dotsc,x_{ne}$. Each contributing term in the expression $E_{ne-je} (0,\dotsc,0, a_2-a,\dotsc,a_2-a, \dotsc, a_n-a,\dotsc, a_n-a)$ is of the form $(a_{t_1}-a)\dotsc (a_{t_{ne-je}}-a)$, where $a_{t_1},\dotsc, a_{t_{ne-je}} \in \{ a,a_2,\dotsc,a_n \}$. For all $t>j$ we have that $\overline{\nu}(a-a_t) < \g$ and consequently $\overline{\o}(x-a_t) = \overline{\nu}(a-a_t)$. If $\{ a_{t_1},\dotsc, a_{t_{ne-je}} \} = \{ a_{j+1},\dotsc,a_n \}$ with each term appearing $e$ times, then
	\[ \overline{\nu}((a_{t_1}-a)\dotsc (a_{t_{ne-je}}-a)) = e\overline{\nu}((a_{j+1}-a) \dotsc (a_n-a)) = -\o h. \]
	Otherwise, there is some $(a_{t_k}-a)$ whose value is at least $\g$. In that case, we observe that 
	\[ \overline{\nu}((a_{t_1}-a)\dotsc (a_{t_{ne-je}}-a)) > e\overline{\nu}((a_{j+1}-a) \dotsc (a_n-a)). \]
	It then follows from the triangle inequality that 
	\[ \overline{\nu}c_{je} = \overline{\nu}E_{ne-je}(0,\dotsc,0, a_2-a,\dotsc, a_2-a, \dotsc, a_n-a,\dotsc,a_n-a) = -\o h. \]
	Thus
	\begin{equation}\label{eqn hc_{je}omega not zero}
		(hc_{je})\overline{\o}\neq 0.
	\end{equation}
	Now consider some $i>j$. The coefficient of $(g(x-a)^e\overline{\o})^i$ in $(g^jhQ^e)\o$ is given by $(g^{j-i}hc_{ie})\overline{\o}$. We observe that
	\[ c_{ie} = (-1)^{ne-ie} E_{ne-ie} (0,\dotsc,0, a_2-a,\dotsc,a_2-a, \dotsc, a_n-a,\dotsc, a_n-a). \] 
	Take any contributing factor in the expression of $E_{ne-ie}$ of the form $(a_{t_1}-a)\dotsc (a_{t_{ne-ie}-a})$ where $a_{t_1},\dotsc, a_{t_{ne-ie}}\in \{ a,a_2,\dotsc,a_n \}$.
	The facts that $ne-ie<ne-je$ and $\overline{\nu}(a-a_t)<\g$ for all $t>j$ imply that
	\[ \overline{\nu}((a_{t_1}-a)\dotsc (a_{t_{ne-ie}-a})) + (ie-je)\g > e\overline{\nu}((a_{j+1}-a) \dotsc (a_n-a)). \]
	It follows that $\overline{\nu}c_{ie}+ (ie-je)\g +\o h>0$ for all $i$. Consequently,
	\begin{equation}\label{eqn g^{j-i}hc_{ie}omega is zero  }
		(g^{j-i}hc_{ie})\overline{\o} =0 \text{ for all } i>j.
	\end{equation}
	We have thus obtained from (\ref{eqn hc_{je}omega not zero}) and (\ref{eqn g^{j-i}hc_{ie}omega is zero  }) that 
	\begin{equation}
		\deg (g^jhQ^e)\o = j.
	\end{equation}
	In light of (\ref{eqn K(x) res field}), (\ref{eqn K(a,x) res field}) and (\ref{eqn g^jhQ^e omega in polynomial}) we observe that
	\[ [K(a,x)\overline{\o}:K(x)\o] = \deg (g^jhQ^e)\o = j. \]	
	The assumption $\g\leq\kras(a,K)$ implies that $j>1$. Hence $K(x)\o\neq K(a,x)\overline{\o}$ and consequently $a\notin K(x)^h$. It now follows from (\ref{eqn r.t. IC and min fields}) that $IC(K(x)|K,\o)\subsetneq K(a)^h$.	
	\end{proof}

\begin{Corollary}\label{corollary gamm < kras}
	Take a minimal pair of definition $(a,\g)$ for $\o$. Assume that $a$ is separable over $K$, $[K(a):K]$ is a prime number and that there is a unique extension of $\nu$ from $K$ to $K(a)$. Then $IC(K(x)|K,\o) = K^h$ if and only if $\g\leq\kras(a,K)$. Otherwise, $IC(K(x)|K,\o) = K(a)^h$. 
\end{Corollary}

\begin{proof}
	The fact that $\nu$ admits a unique extension to $K(a)$ implies that $K^h$ and $K(a)$ are linearly disjoint over $K$ [Lemma 2.1, \ref{Kuh max imm extns of valued fields}]. Thus $[K(a)^h:K^h]=[K(a):K]$ is a prime number. It follows from Theorem \ref{Thm gamma and kras}$(ii)$ and Theorem \ref{Thm gamma and kras}$(iii)$ that $K^h \subseteq IC(K(x)|K,\o) \subsetneq K(a)^h$ if $\g\leq \kras(a,K)$. The fact that $[K(a)^h:K^h]$ is prime then implies that $IC(K(x)|K,\o) = K^h$. If $\g>\kras(a,K)$, then we have $IC(K(x)|K,\o) = K(a)^h$ by Theorem \ref{Thm gamma and kras}$(i)$. We thus have the result.  
\end{proof}

We can employ Theorem \ref{Thm gamma and kras} to give a satisfactory description of $IC(K(x)|K,\o)$ when $(K,\nu)$ is either henselian or has rank one. We first prove the following lemma. 

\begin{Lemma}\label{Lemma key pol irr over heselian or rank one}
	Assume that $(K,\nu)$ is either henselian or has rank one. Further, assume that $\nu K$ is cofinal in $\o K(x)$. Take a key polynomial $Q(x)$ for $\o$ over $K$. Then $Q(x)$ is irreducible over $K^h$.
\end{Lemma}

\begin{proof}
	A key polynomial is irreducible by definition, hence the assertion is trivial when $(K,\nu)$ is henselian. We now consider that $(K,\nu)$ has rank one. Suppose that $Q$ is reducible over $K^h$. Consider a decomposition $Q = fg$ where $f,g \in K^h[x]$ and $\deg f, \deg g \geq 1$. Take a maximal root $a$ of $Q$. Without any loss of generality, we can assume that $f(a)=0$. It follows that
	\[ \d(f)=\d(Q). \]
	The fact that $\nu K$ is cofinal in $\o K(x)$ implies that $\nu K$ is cofinal in $\overline{\o}\overline{K}(x)$. Hence there exists $\a\in\nu K$ such that 
	\[ \a>\max\{0,\d(Q) \}. \]
	Write $f(x)$ as 
	\[ f(x) = \sum_{i=1}^{m} b_i x^i = (x-a_1)\dotsc (x-a_m), \]
	where $b_i \in K^h$ and $a_i\in\overline{K}$. At this point, we invoke the theorem of \enquote{continuity of roots} [cf. \ref{Kuh book}]. Consider a polynomial $f^\prime (x)\in K^h[x]$ given by
	\[ f^\prime(x) = \sum_{i=1}^{m} b_i^\prime x^i = (x-a_1^\prime)\dotsc (x-a_m^\prime). \]
	It follows from the continuity of roots that there exists $\b\in\nu K$ such that $\overline{\nu}(a_i-a_i^\prime) > \a$ for all $i$ whenever $\overline{\nu}(b_i-b_i^\prime) > \b$ for all $i$. Now, the fact that $(K,\nu)$ has rank one implies that $K^h$ is a subfield of the completion $\hat{K}$. Hence we can choose $f^\prime(x)\in K[x]$ with $b_i^\prime\in K$ such that $\overline{\nu}(b_i-b_i^\prime) > \b$ for all $i$. It follows that $\overline{\nu}(a-a^\prime) > \a>\overline{\o}(x-a)$ and consequently $\overline{\o}(x-a) = \overline{\o}(x-a^\prime)$. We have thus obtained that 
	\[ f^\prime (x)\in K[x] \text{ and } \d(f^\prime)\geq \d(Q).  \] 
	However, $\deg f^\prime = \deg f < \deg Q$. The above observation now contradits the assumption that $Q$ is a key polynomial for $\o$ over $K$. It follows that $Q$ is irreducible over $K^h$.
\end{proof}

Let assumptions be as in Lemma \ref{Lemma key pol irr over heselian or rank one}. Recall from Theorem \ref{Thm min pair of definition with sep a} that we can always choose a minimal pair of definition $(a,\g)$ for $\o$ with $a$ separable over $K$. Take the minimal polynomial $Q(x)$ of $a$ over $K$. It follows from Corollary \ref{corollary to Novacoski} that $Q$ is a key polynomial for $\o$ over $K$. The preceding lemma then implies that $Q$ is irreducible over $K^h$, that is, $K(a)$ and $K^h$ are linearly disjoint over $K$. It follows from Lemma 2.1 of [\ref{Kuh max imm extns of valued fields}] that there is a unique extension of $\nu$ from $K$ to $K(a)$. The following result is then a direct consequence of Theorem \ref{Thm gamma and kras}:

\begin{Corollary}\label{Corollary IC when K henslian or rank one}
	Let assumptions be as in Lemma \ref{Lemma key pol irr over heselian or rank one}. Take a minimal pair of definition $(a,\g)$ for $\o$ such that $a$ is separable over $K$. Then 
	\[ IC(K(x)|K,\o) = K(a)^h \text{ if and only if } \g>\kras(a,K). \]
\end{Corollary}


\section{Other results}\label{section other results}

\begin{Proposition}\label{Proposition IC = K^h}
	Take a minimal field of definition $K(a)$ for $\o$. Assume that $a$ is purely inseparable over $K$. Then $IC(K(x)|K,\o) = K^h$.
\end{Proposition}

\begin{proof}
	Take $\g\in\overline{\o}\overline{K}(x)$ such that $(a,\g)$ is a pair of definition for $\o$. Then $(a,\g)$ is also a pair of definition for $\overline{\o}$ over $K^h$. It follows from Lemma \ref{Lemma K(a,x)|K(a) is weakly pure} that $(K^h(a,x)|K^h(a),\overline{\o})$ is a weakly pure extension and that $IC(K^h(a,x)|K^h(a),\overline{\o}) = K^h(a)$. As a consequence, $IC(K^h(x)|K^h,\overline{\o})\subseteq K^h(a)$. Now, $a$ being purely inseparable over $K$ is again purely inseparable over $K^h$. Consequently, $K^h(a)$ is a purely inseparable extension of $K^h$. On the other hand, $IC(K^h(x)|K^h,\overline{\o})$ is a separable-algebraic extension of $K^h$. Thus $IC(K^h(x)|K^h,\overline{\o}) \sect K^h(a) = K^h$. The fact that $IC(K^h(x)|K^h,\overline{\o})\subseteq K^h(a)$ then implies that $IC(K^h(x)|K^h,\overline{\o}) = K^h$. We further observe that $K^h(x) \subset K(x)^h \subseteq (K^h(x))^h$. The fact that $K(x)^h$ is henselian then implies that $K(x)^h = (K^h(x))^h$. We thus have the relations: 
	\[ IC(K(x)|K,\o) = \overline{K}\sect K(x)^h = \overline{K}\sect (K^h(x))^h = IC(K^h(x)|K^h,\overline{\o}) = K^h. \]
\end{proof}

Purely inseparable extensions are examples of purely wild extensions. In Example \ref{example pur wild IC = K(a)} we will illustrate that the conclusion of Proposition \ref{Proposition IC = K^h} may fail to hold when $K(a)|K$ is not purely inseparable, even when $(K(a)|K,\overline{\nu})$ is a purely wild extension. Further, the converse to Proposition \ref{Proposition IC = K^h} is not true. Namely, one can find a valued field $(K,\nu)$ with a valuation transcendental extension $\o$ to $K(x)$ such that $IC(K(x)|K,\o) = K^h$ and $\o$ admits a separable-algebraic minimal field of definition which is a proper extension of $K$. An example illustrating this fact is provided in Example \ref{example pur wild IC =K}.

The next result shows that we can provide a complete solution to Question \ref{question} whenever $\o$ is a value transcendental extension with a unique pair of definition.

\begin{Theorem}\label{Thm gamma > vK}
	Take a minimal pair of definition $(a,\g)$ for $\o$. Assume that $\g>\overline{\nu}\overline{K}$. Then
	\[ IC(K(x)|K,\o) = (K(a)^h|K)^\sep. \]
\end{Theorem}

\begin{proof}
	We first assume that $a$ is purely inseparable over $K$. Then $K^h(a)|K^h$ is a purely inseparable extension. Consequently, $K^h$ is the separable closure of $K$ in $K(a)^h$. The assertion now follows from Proposition \ref{Proposition IC = K^h}. 
	
	We now assume that $a$ is not purely inseparable over $K$. The henselization being a separable extension implies that the separable closure of $K$ in $K(a)^h$ is the same as the separable closure of $K^h$ in $K(a)^h$. Now, the fact that $K(a)^h|K^h$ is a finite extension implies that $(K(a)^h|K^h)^\sep$ is a finite separable extension of $K^h$ and thus simple. Take $f(x)\in K^h[x]$ such that 
	\[ (K(a)^h|K^h)^\sep = K^h(f(a)). \]
	Write $f(x) = c_0+c_1 x+\dotsc + c_n x^n$ where $c_i \in K^h$. Then
	\[ f(x)-f(a) = c_1 (x-a)+\dotsc +c_n (x^n-a^n) = d(x-a)(x-a_2)\dotsc (x-a_n), \]
	where $d\in K^h$ and $a_i \in\overline{K}$. The fact that $\g>\overline{\nu}\overline{K}$ implies that $(a,\g)$ is the unique pair of definition for $\o$. It follows that $\overline{\o}(f(x)-f(a))=\g+\overline{\nu}c$ for some $c\in\overline{K}$. For the same reason we also have that $\overline{\o}(f(x)-f(a))>\overline{\nu}\overline{K}$. Consequently,
	\[ \overline{\o}(f(x)-f(a)) > \kras(f(a),K). \]
	We can then apply Lemma \ref{Lemma Krasner} to conclude that
	\[ f(a)\in K(f(x))^h. \]
    It follows that 
	\[  K^h(f(a)) \subseteq K(f(x))^h \subseteq K(x)^h. \]
	As a consequence, we have from (\ref{eqn IC and min fields}) that
	\[ K^h(f(a))\subseteq IC(K(x)|K,\o)\subseteq K(a)^h.  \]
	The facts that $K^h(f(a))$ is the separable closure of $K$ in $K(a)^h$ and $IC(K(x)|K,\o)$ is a separable-algebraic extension over $K$ now imply that
	\[ IC(K(x)|K,\o) = K^h(f(a)) = (K(a)^h|K)^\sep. \]
\end{proof}

Take a minimal field of definition $K(a)$ for $\o$ and assume that $a$ is not purely inseparable over $K$. Observe that $IC(K(x)|K,\o)$ is contained in the separable closure of $K$ in $K(a)^h$. When $(K,\nu)$ is henselian and $a$ is inseparable over $K$, then an example is provided in Example \ref{example insep IC} where the implicit constant field equals the separable closure of $K$ in $K(a)$, even without the restrictions of Theorem \ref{Thm gamma > vK}. In Example \ref{example sep IC proper nontrivial}, we will construct an example where $a$ is separable over the henselian field $(K,\nu)$, and the implicit constant field is a proper non-trivial subextension of $K(a)|K$.


\section{Examples}\label{section examples} 
 
 \begin{Lemma}\label{Lemma v a < gamma}
 	Take a pair of definition $(a,\g)$ for $\o$. Let $\overline{\nu} a$ has order $e$ modulo $\nu K$. Assume that $\overline{\nu} a<\g$ and $[K(a):K]=e$. Then $(a,\g)$ is a minimal pair of definition for $\o$. 
 \end{Lemma}
 
 \begin{proof}
 	Take another pair of definition $(a^\prime,\g)$ for $\o$. Then $\overline{\nu}(a-a^\prime)\geq\g>\overline{\nu}a$ and hence $\overline{\nu}a=\overline{\nu}a^\prime$. Thus $\overline{\nu}a^\prime$ has order $e$ modulo $\nu K$. It now follows from the fundamental inequality that $[K(a^\prime):K]\geq e$. Thus $(a,\g)$ is a minimal pair of definition for $\o$.
 \end{proof}

\begin{Example}\label{example pur wild IC = K(a)}
	Let $(K,\nu)$ be the valued field $\QQ$ equipped with the $p$-adic valuation, where $p$ is a non-zero prime. Then $\nu K=\ZZ$ and $K\nu=\FF_p$. Fix an extension of $\nu$ to $\overline{K}$ and denote it again by $\nu$. Take $a\in\overline{K}$ such that $a^p=\frac{1}{p}$. Hence $\nu a = -\frac{1}{p}$ and thus $\nu a$ has order $p$ modulo $\nu K$. It follows that $[K(a):K]=p=(\nu K(a):\nu K)$. The conjuagates of $a$ are $\{ \zeta^i a \}_{i=0}^{p-1}$ where $\zeta$ is a primitive $p$-th root of unity. It is well known that $\nu(1-\zeta^i) = \frac{1}{p-1}$ for all $i\in\{ 1,\dotsc,p-1 \}$. Thus 
	\[ \kras(a,K) = \frac{1}{p-1} - \frac{1}{p}. \]
	Take a real number $\g\in\RR$ such that $\g>\kras(a,K)$. Hence $\g>\nu a$. Take the extension $\overline{\o}$ of $\nu$ to $\overline{K}(x)$ given by $\overline{\o}:= \nu_{a,\g}$. Then $\o:= \overline{\o}|_{K(x)}$ is a valuation transcendental extension of $\nu$ to $K(x)$ with $(a,\g)$ as a pair of definition. In light of Lemma \ref{Lemma v a < gamma} we have that $(a,\g)$ is a minimal pair of definition for $\o$. It now follows from Theorem \ref{Thm gamma and kras}$(i)$ that $K(a)^h=IC(K(x)|K,\o)$.
	
	We can also start with the henselization of $(K,\nu)$ without any change in conclusions. The fact that $[K(a):K] = p = (\nu K(a):\nu K)$ implies that there is a unique extension of $\nu$ from $K$ to $K(a)$. In light of Lemma 2.1 of [\ref{Kuh max imm extns of valued fields}] we conclude that $(K^h(a)|K^h,\nu)$ is a purely wild extension of degree $p$ and $\kras(a,K^h) = \kras(a,K)$. Further, it follows from Lemma \ref{Lemma v a < gamma} that $(a,\g)$ is a minimal pair of definition for $\overline{\o}$ over $K^h$. In light of Theorem \ref{Thm gamma and kras}$(i)$, we finally obtain that $IC(K^h(x)|K^h,\overline{\o}) = K^h(a)$.    
\end{Example}

\begin{Example}\label{example pur wild IC =K}
	Take $K,\nu,p$ and $a$ as in Example \ref{example pur wild IC = K(a)}. Observe that $\nu a<\kras(a,K)$. Take a real number $\g$ such that $\nu a<\g\leq\kras(a,K)$. Consider the extension $\overline{\o}:= \overline{\nu}_{a,\g}$ and set $\o:= \overline{\o}|_{K(x)}$. Then $\o$ is a valuation transcendental extension of $\nu$ to $K(x)$. $\o$ is value transcendental when $\g$ is irrational, and $\o$ is residue transcendental otherwise. It follows from Lemma \ref{Lemma v a < gamma} that $(a,\g)$ is a minimal pair of definition for $\o$. We can then conclude from Theorem \ref{Thm gamma and kras} that $IC(K(x)|K,\o) = K^h$.
\end{Example}

\begin{Example}\label{example sep IC proper nontrivial}
	Take an odd prime $p$. Consider the valued field $\FF_p(t)$ equipped with the $t$-adic valuation $\nu=\nu_t$. Fix an extension of $\nu$ to $\overline{\FF_p (t)}$ and denote it again by $\nu$. Set $K:= \FF_p(t)^h$. Then $\nu K=\ZZ$ and $K\nu=\FF_p$.
	\newline Consider the polynomials $f(x):= x^p-x-\frac{1}{t}$ and $g(x):= x^2-\frac{1}{t}$. Take $a_1, a_2 \in\overline{K}$ such that $f(a_1)=0=g(a_2)$. Observe that $\nu a_1 = -\frac{1}{p}$ and $\nu a_2 = -\frac{1}{2}$. The following facts are now evident in light of the fundamental inequality:
	\begin{align*}
		[K(a_1):K]=p = (\nu K(a_1):\nu K) &\text{ and } \nu K(a_1) = \frac{1}{p}\ZZ.\\
		[K(a_2):K]=2 = (\nu K(a_2):\nu K) &\text{ and } \nu K(a_2) = \frac{1}{2}\ZZ.
	\end{align*}
	Thus $K(a_1)$ and $K(a_2)$ are linearly disjoint over $K$ and hence $[K(a_1,a_2):K] =2p$. Further, the fact that $\nu K(a_1)$ and $\nu K(a_2)$ are subgroups of $\nu K(a_1,a_2)$ implies that $\frac{1}{p}\ZZ\dirsum\frac{1}{2}\ZZ \subseteq \nu K(a_1,a_2)$. As a consequence, we obtain that
	\[ [K(a_1,a_2):K] = 2p = (\nu K(a_1,a_2):\nu K) \text{ and } \nu K(a_1,a_2) = \frac{1}{2p}\ZZ. \]
	Take $a\in K(a_1,a_2)$ such that $\nu a = \frac{1}{2p}$. It then follows from the fundamental inequality that $2p\leq (\nu K(a):\nu K) \leq [K(a):K]\leq [K(a_1,a_2):K]=2p$. Hence,
	\[ K(a)=K(a_1,a_2) \text{ and } [K(a):K]=2p=(\nu K(a):\nu K). \]
	We observe that $a_1$ and $a_2$ are both separable over $K$. Thus $K(a_1,a_2)$ is a separable extension of $K$ and hence $a$ is separable over $K$. Suppose that $\nu a =\kras (a,K)$. It then follows from Proposition 5.10 of [\ref{Kuh value groups residue fields rational fn fields}] that $(K(a)|K,\nu)$ is a tame extension. However, this is not possible since $p$ divides $(\nu K(a):\nu K)$. It follows that
	\[ \nu a < \kras(a,K). \]
	Take an irrational number $\g\in\RR$ such that
	\[ \nu a<\g<\kras(a,K). \]
	Consider the valuation $\overline{\o}$ on $\overline{K}(x)$ given by $\overline{\o}:= \nu_{a,\g}$ and set $\o:=\overline{\o}|_{K(x)}$. Then $(a,\g)$ is a pair of definition for $\o$. We observe from Lemma \ref{Lemma v a < gamma} that $(a,\g)$ is a minimal pair of definition for $\o$. It follows from Theorem \ref{Thm gamma and kras}$(ii)$ that $IC(K(x)|K,\o) \subsetneq K(a)$. Hence, 
	\[ [IC(K(x)|K,\o):K]<[K(a):K]=2p. \]
	Now, the fact that $[K(a_2):K]=2 = (\nu K(a_2):\nu K)$ implies that $(K(a_2)|K,\nu)$ is a tame extension. It follows from Theorem \ref{Thm IC and min fields} that $K(a_2)\subseteq IC(K(x)|K,\o)$. Thus $2$ divides $[IC(K(x)|K,\o):K]$. The fact that $[IC(K(x)|K,\o):K]$ divides $2p$, coupled with the previous observations, now imply that $[IC(K(x)|K,\o):K] =2$. It thus follows that
	\[ IC(K(x)|K,\o) = K(a_2). \]
\end{Example}

\begin{Example}\label{example insep IC}
	Take a field $k$ with $\ch k=p>0$ and let $u,v$ be indeterminates over $k$. Define the map $\nu: k[u,v] \longrightarrow (\ZZ\dirsum\ZZ)_{\text{lex}}$ given by
	\[ \nu(\sum c_{i,j} u^i v^j ) := \min\{ (i,j) \mid c_{i,j}\neq 0 \}, \]
	and extend $\nu$ canonically to $k(u,v)$. Then $\nu$ is valuation on $k(u,v)$ with $\nu k(u,v) = (\ZZ\dirsum\ZZ)_{\text{lex}}$. Fix an extension of $\nu$ to $\overline{k(u,v)}$ and denote it again by $\nu$. Set $K:= k(u,v)^h$. Consider the polynomials
	\begin{align*}
		f(x) &:= x^{p^2} + ux^p + v,\\
		g(x) &:= x^p + ux +v.
	\end{align*}
	Take $a\in\overline{K}$ such that $f(a)=0$. It follows from the triangle inequality that at least two monomials in the expression $f(a) = a^{p^2} + ua^p + v$ have the same value. Suppose that $\nu a^{p^2} = \nu (ua^p)$. Then $\nu a = (\frac{1}{p^2-p},0)$ and consequently, $\nu a^{p^2} = \nu(ua^p) > \nu v$. It then follows from the triangle inequality that $\nu v = \nu f(a)$ which contradicts the fact that $f(a)=0$. Now suppose that $\nu(ua^p)=\nu v$. Then $\nu a= (\frac{-1}{p},\frac{1}{p})$ and consequently, $\nu a^{p^2} = (-p,p) < \nu v= \nu(ua^p)$. Thus $\nu f(a) = (-p,p)$ which again yields a contradiction. It then follows that $\nu a^{p^2} = \nu v$, that is, $\nu a = (0,\frac{1}{p^2})$. Hence $\nu a$ has order $p^2$ modulo $\nu K$. It now follows from the fundamental inequality that $(\nu K(a):\nu K) = p^2 = [K(a):K]$. As a consequence, we obtain that $f(x)$ is irreducible over $K$. Similar arguments also yield that $g(x)$ is irreducible over $K$.
	\newline Now, $K(a)|K$ is an inseparable extension of degree $p^2$. Further, $g(a^p)=0$ and thus $K(a^p)|K$ is a separable extension of degree $p$. It follows that $K(a^p)$ is the separable closure of $K$ in $K(a)$.
	
	\pars Consider the polynomial 
	\[ h(x) := x^p + ux.  \]
	Take a root $t$ of $h(x)$. Then,
	\[ t^p = -ut \Longrightarrow (p-1)\nu t = \nu u\Longrightarrow \nu t = (\frac{1}{p-1},0). \]
	 Observe that any root of $g(x)$ is of the form $b^p$ for some root $b$ of $f$. Take two distinct roots $b_1$ and $b_2$ of $f$. Then $h(b_1^p-b_2^p) =0$ and hence $\nu(b_1^p-b_2^p) = (\frac{1}{p-1},0)$. It follows that
	 \[ \kras(a^p,K) = (\frac{1}{p-1},0). \] 
	Choose an irrational number $r\in\RR$ such that $r>\frac{1}{p-1}$. Set $\g:= (r,0)$. Then
	\[ \g> \kras(a^p,K) \]
	under the natural extension of the ordering on $\QQ\dirsum\QQ$ to $\RR\dirsum\QQ$. Take the valuation $\overline{\o}:= \nu_{a,\g}$ and set $\o:= \overline{\o}|_{K(x)}$. The fact that $\g\notin\QQ\dirsum\QQ$ implies that $\o$ is a value transcendental extension of $\nu$ to $K(x)$ with $(a,\g)$ as a pair of definition. Observe that $\g = (r,0)>(0,\frac{1}{p^2}) =\nu a$. It now follows from Lemma \ref{Lemma v a < gamma} that 
	\[ (a,\g) \text{ is a minimal pair of definition for }\o. \]
	It follows from (\ref{eqn IC and min fields}) that $IC(K(x)|K,\o) \subseteq K(a)$. Now, $IC(K(x)|K,\o)$ is a separable-algebraic extension of $K$ while $K(a)|K$ is inseparable. Hence $IC(K(x)|K,\o)\neq K(a)$. We further observe that
	\[ \overline{\o}(x^p-a^p) =p\overline{\o}(x-a)=p\g>\g>\kras(a^p,K). \]
	We can thus apply Lemma \ref{Lemma Krasner} to obtain that $a^p \in K(x^p)^h \subset K(x)^h$. It follows that 
	\[ K\subsetneq K(a^p) \subseteq IC(K(x)|K,\o)\subsetneq K(a). \]
	The fact that $K(a^p)$ is the separable closure of $K$ in $K(a)$ then implies that $K(a^p)=IC(K(x)|K,\o)$.
\end{Example}

The following result is proved in Proposition 3.14 of [\ref{Kuh value groups residue fields rational fn fields}]: Let $(K(a)|K,\overline{\nu})$ be a separable algebraic extension of valued fields and let $\G:= \overline{\nu} K(a)\dirsum\ZZ$ be an abelian group endowed with any extension of the ordering of $\overline{\nu} K(a)$. Then there exists an extension $\overline{\o}$ of $\overline{\nu}$ to $\overline{K(x)}$ such that 
\begin{equation}\label{eqn conditions}
	\o K(x) = \G, \, K(x)\o = K(a)\overline{\nu} \text{ and } IC(K(x)|K,\o) = K(a)^h,
\end{equation}
where $\o = \overline{\o}|_{K(x)}$. Observe that $\o$ is a value transcendental extension of $\nu$ to $K(x)$. It is natural to inquire whether any extension $\o$ satisfying the conditions (\ref{eqn conditions}) has $K(a)$ as a minimal field of definition. The following example illustrates that this is not necessarily true. The first part of the construction follows an example due to F. K. Schmidt [cf. Example 3.1, \ref{Kuh defect}].

\begin{Example}
	Take a non-zero prime number $p$. Consider the field $\FF_p((t))$ equipped with the $t$-adic valuation $\nu=\nu_t$. Extend $\nu$ to $\overline{\FF_p((t))}$ and denote the extension also by $\nu$. Observe that $(\FF_p((t))|\FF_p(t),\nu)$ is an immediate extension with $\nu\FF_p(t)=\ZZ$. Further, the field extension $\FF_p((t))|\FF_p(t)$ has infinite transcendence degree. Choose $\zeta\in\FF_p((t))$ which is transcendental over $\FF_p(t)$. Then $(\FF_p(t,\zeta)|\FF_p(t,\zeta^p),\nu)$ is an immediate purely inseparable extension. Taking henselizations, we observe that $(\FF_p(t,\zeta)^h|\FF_p(t,\zeta^p)^h,\nu)$ is also an immediate purely inseparable extension. Set $K:= \FF_p(t,\zeta^p)^h$. Then $K(\zeta) = \FF_p(t,\zeta)^h$. Take the ordered abelian group $\ZZ\dirsum\QQ$ equipped with the lexicographic order and set $\g:= (1,0)$. Then $\g> \nu\overline{K}$. Take the extension $\overline{\o}:= \nu_{\zeta,\g}$ of $\nu$ to $\overline{K}(x)$ and fix an extension of $\overline{\o}$ to $\overline{K(x)}$. Set $\o:= \overline{\o}|_{K(x)}$. The fact that $\g>\nu\overline{K}$ implies that $(\zeta,\g)$ is the unique pair of definition for $\o$. Now, $\zeta$ is purely inseparable over $K$. We thus obtain from Proposition \ref{Proposition IC = K^h} that $IC(K(x)|K,\o)=K$. Further, it follows from Remark \ref{Remark value group res field} that $\o K(x) = \ZZ\o(x^p-\zeta^p) \dirsum\nu K(\zeta)$ and $K(x)\o = K(\zeta)\nu$. The fact that $(K(\zeta)|K,\nu)$ is an immediate extension implies that $\o K(x) = \ZZ\o (x^p-\zeta^p)\dirsum\nu K$ and $K(x)\o =K\nu$. Take any $a\in K$. Then the relations given in (\ref{eqn conditions}) are satisfied. However, the fact that $(\zeta,\g)$ is the unique pair of definition for $\o$ implies that $K(a)$ is not a minimal field of definition for $\o$.
\end{Example}


\section{Connection with pseudo-Cauchy sequences}\label{section pseudo cauchy}
	
We now explore the relationship between valuation transcendental extensions and pseudo-Cauchy sequences. Recall that we are working under the assumtions that $\o$ is a valuation transcendental extension of $\nu$ to $K(x)$, $\overline{\nu}$ is an extension of $\nu$ to $\overline{K}$ and $\overline{\o}$ is a common extension of $\o$ and $\overline{\nu}$ to $\overline{K}(x)$. We first mention the following general result.

\begin{Lemma}\label{Lemma Kuh obskeypoly}
	Take a pseudo-Cauchy sequence $(z_\mu)_{\mu < \l}$ in $(K,\nu)$ without a limit in $K$. Then the following statements hold true:
	\sn (i) Assume that $f(x)\in K[x]$ is such that $(\nu f(z_\mu))_{\mu<\l}$ is ultimately monotonically increasing. Then at least one root of $f$ is a limit of $(z_\mu)_{\mu<\l}$.
	\n (ii) Assume that $c\in\overline{K}$ is a limit of $(z_\mu)_{\mu<\l}$. Take the minimal polynomial $g(x)$ of $c$ over $K$. Then $(\nu g(z_\mu))_{\mu<\l}$ is ultimately monotonically increasing. 
\end{Lemma}

\begin{proof}
	Write $f(x) = b\prod_{i=1}^{n} (x-a_i)$ where $a_i\in\overline{K}$ and $b\in K$. Then $\nu f(z_\mu) = \nu b +\sum_{i=1}^{n} \overline{\nu}(z_\mu-a_i)$. If none of the $a_i$ is a limit of $(z_\mu)_{\mu<\l}$, then it follows from Proposition \ref{Proposition Kaplansky} that $(\overline{\nu}(a_i-z_\mu))_{\mu<\l}$ is ultimately constant for all $i$. It follows that $(\nu f(z_\mu))_{\mu<\l}$ is ultimately constant which is a contradiction. So at least one root of $f$ is a limit of $(z_\mu)_{\mu<\l}$.
	
	\pars We now assume that $c\in\overline{K}$ is a limit of $(z_\mu)_{\mu<\l}$. By definition, $(\overline{\nu}(c-z_\mu))_{\mu<\l}$ is a monotonically increasing sequence. Write $g(x) = \prod_{i=1}^{n} (x-c_i)$. Then $\nu g(z_\mu) = \sum_{i=1}^{n} \overline{\nu}(z_\mu-c_i)$. We observe from Proposition \ref{Proposition Kaplansky} that $(\overline{\nu}(c_i-z_\mu))_{\mu<\l}$ is monotonically increasing whenever $c_i$ is a limit of $(z_\mu)_{\mu<\l}$, and the sequence is ultimately constant otherwise. It follows that $(\nu g(z_\mu))_{\mu<\l}$ is ultimately monotonically increasing. 
\end{proof}

\begin{Theorem}
	Take a pseudo-Cauchy sequence $(z_\mu)_{\mu < \l}$ in $(K,\nu)$ without a limit in $K$. Assume that $x$ is a limit of $(z_\mu)_{\mu < \l}$. Then the following statements hold true:
	\sn (i) $(z_\mu)_{\mu < \l}$ is a pseudo-Cauchy sequence of algebraic type.
	\n (ii) Whenever $(a,\g)$ is a pair of definition for $\o$, we have that $a$ is a limit of $(z_\mu)_{\mu < \l}$.
	\n (iii) Let $Q(x)\in K[x]$ be a nonlinear key polynomial for $\o$ with a maximal root $b\in\overline{K}$. Then $b$ is also a limit of $(z_\mu)_{\mu < \l}$.
	\n (iv) Take an associated minimal polynomial $m(x) \in K[x]$ of $(z_\mu)_{\mu < \l}$. Then $m(x)$ is a key polynomial for $\o$.
	\pars Conversely, assume that $(a,\g)$ is a minimal pair of definition for $\o$ and that $a$ is a limit of $(z_\mu)_{\mu < \l}$. Then $x$ is also a limit of $(z_\mu)_{\mu < \l}$.
\end{Theorem}

\begin{proof}
	Set $\g_\mu := \nu(z_\mu - z_{\mu +1})$ for each $\mu < \l$. By definition, $(\g_\mu)_{\mu < \l}$ is a strictly increasing sequence. If $x$ is a limit of the pseudo-Cauchy sequence $(z_\mu)_{\mu < \l}$, then by definition $\o(x-z_\mu) = \g_\mu$ for each $\mu<\l$. If $(z_\mu)_{\mu < \l}$ is of transcendental type, then the extension $(K(x)|K,\o)$ is immediate in view of Theorem 2 of [\ref{Kaplansky}]. Hence $(z_\mu)_{\mu < \l}$ is of algebraic type. 
	
	\pars Let $(a,\g)$ be a pair of definition for $\o$. Then $\overline{\o}(x-a) = \g = \max \overline{\o}(x-\overline{K})$. The fact that $(\g_\mu)_{\mu < \l}$ is a monotonically increasing sequence implies that either $\g>\g_\mu$ for each $\mu<\l$, or ultimately $\g < \g_\mu$. Observe that $\g_\mu \in \o(x-K) \subseteq \overline{\o}(x-\overline{K})$ for each $\mu$. The maximality of $\g$ then implies that $\g>\g_\mu$ for each $\mu<\l$. From the triangle inequality it follows that $\overline{\nu}(a-z_\mu) = \g_\mu$ for each $\mu<\l$. Hence $a$ is a also a limit of $(z_\mu)_{\mu < \l}$.
	
	\pars Now let $Q(x)\in K[x]$ be a nonlinear key polynomial for $\o$ with a maximal root $b$, that is, $\d(Q) = \overline{\o}(x-b)$. It follows from Theorem 1.1 of [\ref{Novacoski key poly and min pairs}] that for any $c\in\overline{K}$,
	\[ \overline{\nu}(b-c) \geq \d(Q) \Longrightarrow [K(c):K]\geq [K(b):K]. \]
	Suppose that $\g_\mu \geq \d(Q)$ for some $\mu<\l$. Then, $\o(x-z_\mu) \geq \overline{\o}(x-b)$. Consequently, we have that $\overline{\nu} (b-z_\mu) \geq \overline{\o}(x-b) = \d(Q)$. It follows from our preceding discussions that $[K(z_\mu):K] \geq [K(b):K]\geq 2$ which is a contradiction. Thus $\g_\mu < \d(Q)$ for all $\mu<\l$. From the triangle inequality it follows that $\overline{\nu}(b-z_\mu) = \g_\mu$ for each $\mu<\l$ and hence $b$ is also a limit of $(z_\mu)_{\mu < \l}$.
	
	\pars Now take an associated minimal polynomial $m(x)\in K[x]$ of $(z_\mu)_{\mu < \l}$ with a maximal root $l\in\overline{K}$. The sequence $(\nu m(z_\mu))_{\mu<\l}$ is ultimately monotonically increasing. It follows from Lemma \ref{Lemma Kuh obskeypoly} that there exists some root $l^\prime$ of $m$ such that $l^\prime$ is a limit of $(z_\mu)_{\mu < \l}$. Thus, 
	\[ \d(m) = \overline{\o}(x-l) \geq \overline{\o}(x-l^\prime) > \g_\mu \text{ for all } \mu<\l  \]
	and hence $l$ is also a limit of $(z_\mu)_{\mu < \l}$. Take $c\in\overline{K}$ such that $\overline{\nu}(c-l)\geq \d(m)$. Then $c$ is also a limit of $(z_\mu)_{\mu < \l}$. Take the minimal polynomial $g(x)\in K[x]$ of $c$ over $K$. In light of Lemma \ref{Lemma Kuh obskeypoly} we have that $(\nu g(z_\mu))_{\mu<\l}$ is ultimately monotonically increasing and hence $\deg g \geq \deg m$. Thus, 
	\[ \overline{\nu}(c-l) \geq \d(m) \Longrightarrow [K(c):K] \geq [K(l):K]. \]
	It follows from Theorem 1.1 of [\ref{Novacoski key poly and min pairs}] that $m(x)$ is a key polynomial for $\o$.
	
 	\pars Conversely, assume that $(a,\g)$ is a minimal pair of definition for $\o$ and that $a$ is a limit of $(z_\mu)_{\mu < \l}$. Then $\overline{\nu}(a-z_\mu)=\g_\mu$ for each $\mu<\l$. If $\g\leq\g_\mu$ for some $\mu$, then $(z_\mu,\g)$ is also a pair of definition for $\o$. The fact that $z_\mu\in K$  and the minimality of $(a,\g)$ then implies that $a\in K$ which contradicts the fact that $(z_\mu)_{\mu < \l}$ does not admit any limits in $K$. Hence $\g>\g_\mu$ for each $\mu<\l$. It follows from the triangle inequality that $\overline{\nu}(x-z_\mu) = \g_\mu$ for each $\mu<\l$ and hence $x$ is also a limit of $(z_\mu)_{\mu < \l}$.
\end{proof}
 
We can construct a pseudo-Cauchy sequence $(z_\mu)_{\mu < \l}$ in $(K,\nu)$ with $x$ as a limit, such that for some $c\in\overline{K}$, $c$ is also a limit of the sequence $(z_\mu)_{\mu < \l}$ but the minimal polynomial of $c$ is not a key polynomial for $\o$. We illustrate this fact in the next example. We will require the following definition: take an extension of valued fields $(K(z)|K,\nu)$. The \textbf{distance} of $z$ from $K$, denoted by $\dist(z,K)$, is defined to be the least initial segment of $\overline{\nu}\overline{K}$ containing $\nu(z-K)\sect \nu K$. We write $\dist(z,K)=(\a)^-$ for some $\a\in\overline{\nu}\overline{K}$ to mean that $\dist(z,K) = \{ \th\in\overline{\nu}\overline{K}\mid \th < \a \}$.  	
	
\begin{Example}
	Take an odd prime $p$ and let $(k,\nu)$ be the valued field $(\FF_p(t),\nu_t)$ equipped with the $t$-adic valuation. Fix an extension of $\nu$ to $\overline{\FF_p (t)}$ which we again denote by $\nu$. Set $K$ to be the perfect hull of $k^h$. Hence $(K,\nu)$ is also henselian. Further, 
	\[ \nu K = \frac{1}{p^\infty}\ZZ \text{ and } K\nu = \FF_p. \]
	Take $a\in\overline{\FF_p(t)}$ such that $a^p - a  = \frac{1}{t}$. Then $(K(a)|K,\nu)$ is an immediate Artin-Schreier defect extension with $\dist(a,K) = (0)^-$ [Example 3.12, \ref{Kuh defect}]. It follows from [Theorem 1, \ref{Kaplansky}] that $a$ is a limit of a pseudo-Cauchy sequence $(z_\mu)_{\mu < \l}$ in $(K,\nu)$ which does not have any limits in $K$. Set $\g_\mu := \nu(z_\mu - z_{\mu+1})$ for all $\mu<\l$. Then, 
	\[ \nu(a-z_\mu) = \g_\mu<0 \text{ for all } \mu<\l. \]
	Choose an irrational number $\g$ such that $0<\g<\frac{1}{2}$. Take the value transcendental extension $\overline{\o}:= \nu_{a,\g}$ of $\nu$ to $\overline{K}(x)$ and set $\o:= \overline{\o}|_{K(x)}$. Then $\overline{\o}(x-a)=\g>\g_\mu$ for all $\mu<\l$ and hence $x$ is a limit of $(z_\mu)_{\mu < \l}$. Suppose that $(a,\g)$ is not a minimal pair of definition for $\o$. Then there exists $a^\prime\in\overline{K}$ such that $\overline{\nu}(a-a^\prime)>\g$ and $[K(a^\prime):K] < [K(a):K]=p$. Hence $a^\prime$ is a limit of $(z_\mu)_{\mu < \l}$. It further follows from Lemma 2.17 of [\ref{Kuh A-S extensions and defectless fields paper}] that $\dist(a,K) = (0)^- = \dist(a^\prime,K)$. Take the minimal polynomial $m_{a^\prime}(x)$ of $a^\prime$ over $K$. It follows from Lemma 2.22 of [\ref{Kuh A-S extensions and defectless fields paper}] that $\dist(m_{a^\prime}(a),K) = \dist(0,K) = \nu\overline{K}$. As a consequence we have $m_{a^\prime}(a)=0$. However this contradicts the assumption that $[K(a^\prime):K] < [K(a):K]$. Hence we have that $(a,\g)$ is a minimal pair of definition for $\o$. Consequently, $m_a(x)$ is a key polynomial for $\o$ [Corollary \ref{corollary to Novacoski}].
	
	\pars Take $b\in\overline{K}$ such that $b^2 = t$. Thus $[K(b):K]\leq 2$. Now, $\nu b = \frac{1}{2} \notin \nu K = \frac{1}{p^\infty}\ZZ$. It then follows from the fundamental inequality that $[K(b):K] = 2$. Thus $K(a)$ and $K(b)$ are linearly disjoint over $K$ and hence $[K(a,b):K] = 2p$. Set $c := a+b$. Then $\nu(c-a) = \nu b = \frac{1}{2} > \g>0$. Hence $(c, \g)$ is a pair of definition for $\o$ and $c$ is a limit of $(z_\mu)_{\mu<\l}$. Suppose that the corresponding minimal polynomial $m_c(x)$ is a key polynomial for $\o$. The minimality of $(a,\g)$ implies that $\deg m_c\geq \deg m_a$. Consequently, $\d(m_c)\geq \d(m_a)$. The fact that $(a,\g)$ and $(c,\g)$ are pairs of definition for $\o$ imply that $\d(m_c)=\g=\d(m_a)$. We thus have that $\deg m_c = p =\deg m_a$. Now $K(a)|K$ is an Artin-Schreier extension, hence Galois. Thus either $K(a)=K(c)$, or they are linearly disjoint over $K$. It follows that $[K(a,c):K]$ is either $p$ or $p^2$. The observations $K(a,c) = K(a,a+b) = K(a,b)$ and $[K(a,b):K] = 2p$ then imply that we have again arrived at a contradiction. Thus $m_c(x)$ is not a key polynomial of $\o$, while $c$ is a limit of $(z_\mu)_{\mu < \l}$.
\end{Example}

	\end{document}